\documentclass[a4paper,11pt]{amsart}

\usepackage[T1]{fontenc}
\usepackage{latexsym,amsmath,amsfonts,amscd,amssymb,amsthm,mathrsfs}
\usepackage{tikz}
\usepackage{caption}
\usepackage{soul}
\usepackage{pdflscape}
\usepackage{booktabs}
\usepackage{longtable}
\usepackage[symbol]{footmisc}
\usepackage{textcomp}
\usepackage{multirow}
\usepackage{longtable}
\usepackage[hypertexnames=false,
backref=page,
    pdftex,
    pdfpagemode=UseNone,
    breaklinks=true,
    extension=pdf,
    colorlinks=true,
    linkcolor=blue,
    citecolor=blue,
    urlcolor=blue,
]{hyperref}

\usepackage[top=0.9in, bottom=.8in, left=0.55in, right=0.6in]{geometry}
\usepackage{graphicx}
\usepackage{mathtools}

\newcommand{\h}{\mathfrak h}
\newcommand{\g}{\mathfrak g}

\newcommand{\pint}{\langle\cdot,\cdot\rangle}
\newcommand{\I}{\operatorname{Id}}
\newcommand{\ad}{\operatorname{ad}}
\newcommand{\tr}{\operatorname{tr}}
\newcommand{\R}{\mathbb R}

\newcommand{\referenza}{}

\newtheorem{thm}{Theorem}[section]
\newtheorem*{thm*}{Theorem \referenza}

\newtheorem*{cor*}{Corollary \referenza}
\newtheorem{lem}[thm]{Lemma}
\newtheorem*{lem*}{Lemma \referenza}

\newtheorem{prop}[thm]{Proposition}
\newtheorem*{prop*}{Proposition \referenza}

\newtheorem*{conj*}{Conjecture \referenza}
\newtheorem{rmk}[thm]{Remark}

\newtheorem{exa}[thm]{Example}

\numberwithin{equation}{section}

\allowdisplaybreaks[3]

\title[lcK $4$-dimensional Lie algebras]{Locally conformally K\"ahler structures on four-dimensional solvable Lie algebras}

\author{Daniele Angella}
\address[Daniele Angella]{Dipartimento di Matematica e Informatica ``Ulisse Dini''\\
Universit\`a di Firenze\\
via Morgagni 67/A\\
50134 Firenze, Italy}
\email{daniele.angella@gmail.com}
\email{daniele.angella@unifi.it}

\author{Marcos Origlia}
\address[Marcos Origlia]{KU Leuven Kulak Campus Kortrijk\\ E. Sabbelaan 53\\ BE-8500 Kortrijk, Belgium; and FaMAF-CIEM, Universidad Nacional de C\'{o}rdoba \\ 5000 C\'{o}rdoba\\ Argentina}
\email{marcosmiguel.origlia@kuleuven.be}
\email{marcosoriglia@gmail.com}

\thanks{The first-named author is supported by project SIR 2014 AnHyC ``Analytic aspects in complex and hypercomplex geometry'' (code RBSI14DYEB), by project PRIN 2017 ``Real and Complex Manifolds: Topology, Geometry and holomorphic dynamics'' (code 2017JZ2SW5), and by GNSAGA of INdAM. He also acknowledges the support of Istituto Italiano di Cultura in C\'ordoba.
The second-named author is supported by CONICET, SECyTUNC (Argentina) and the Research Foundation Flanders (Project G.0F93.17N)
}

\date{\today}

\begin{document}

\begin{abstract}
We classify and investigate locally conformally K\"ahler structures on four-dimensional solvable Lie algebras up to linear equivalence. As an application we can produce many examples in higher dimension, here including lcK structures on Oeljeklaus-Toma manifolds, and we also give a geometric interpretation of some of the $4$-dimensional structures in our classification.
\end{abstract}

\maketitle

\section*{Introduction}
The aim of this note is to provide explicit examples of locally conformally K\"ahler structures on complex surfaces and higher dimensional manifolds, by classifying left-invariant lcK structures on four-dimensional solvable Lie groups.

A {\em locally conformally K\"ahler} (shortly, {\em lcK}) metric $g$ on a complex manifold $(X,J)$ is a Hermitian metric that locally admits a conformal change $\exp{(-f)} g\lfloor_U$ making it K\"ahler.
Equivalently, the associated $(1,1)$-form $\Omega:=g(J\_,\_)$ satisfies $d\Omega=\theta\wedge\Omega$ where the {\em Lee form} $\theta\stackrel{\text{loc}}{=}df$ is a closed $1$-form.
In other words, one gets a covering endowed with a K\"ahler metric on which the deck transformations group acts by holomorphic homotheties.
One can refer to \cite{dragomir-ornea, ornea-verbitsky-report, bazzoni} and references therein for an open-ended account on lcK geometry: just to cite a few of the several contributions to lcK geometry in the last twenty years, see \cite{ ornea-verbitsky-MRL, ornea-verbitsky-MathAnn, ornea-verbitsky-JGP, gini-ornea-parton, vuletescu, otiman, cappellettimontano-denicola-marrero-yudin, madami-moroianu-pilca, belgun, brunella, pontecorvo-vaisman, gauduchon-moroianu-ornea, achk, bazzoni-BLMS, oeljeklaus-toma, kasuya}.
With the only exception of some Inoue surfaces, every known compact complex surface admits lcK metrics by \cite{belgun, brunella}, see also \cite{pontecorvo-vaisman} for a survey on the Vaisman question.

Among compact complex surfaces, one can describe complex tori, hyperelliptic surfaces, Inoue surfaces of type $S^0$, primary Kodaira surfaces, secondary Kodaira surfaces, and Inoue surfaces of type $S^{\pm}$ as compact quotients of solvable Lie groups endowed with left-invariant complex structures by \cite[Theorem 1]{hasegawa-jsg}.
Complex structures on four-dimensional Lie algebras are classified by \cite{snow-crelle, snow, ovando-manuscr, achk}, see also \cite{ovando}. On the other hand, locally conformally K\"ahler metrics underlie {\em locally conformally symplectic} ({\em lcs}) structures, which are similarly defined. Extending Ovando's results on four-dimensional symplectic Lie algebras \cite{ovando-symplectic}, a classification of four-dimensional locally conformally symplectic Lie algebras is given with structure results in \cite{angella-bazzoni-parton}.

Locally conformally K\"ahler structures on four-dimensional reductive Lie algebras  are studied in \cite[Theorem 4.6]{achk}. In this note, we classify locally conformally K\"ahler structures on four-dimensional {\em solvable} Lie groups. (See \cite{andrada-origlia-CM} for a survey and results on invariant lcK structures on solvmanifolds.) The classification is up to linear equivalence, and complex automorphisms are specified.
The results are summarized in Theorem \ref{thm:classification-lck} and Table \ref{table:lck-solvable}.
We use the classification of complex structures \cite{ovando} and the classification of lcs structures \cite{angella-bazzoni-parton} for four-dimensional solvable Lie algebras, and the computations have been performed with the help of the mathematical open-source software system Sage \cite{sage} (the authors will be happy to share the code with anyone who might be interested).

We are also interested in {\em Vaisman} structures on a $4$-dimensional Lie algebra, that is, lcK structures whose Lee form is parallel with respect to the Levi-Civita connection of the Hermitian structure. Let us recall the following characterization result from \cite{andrada-origlia}.

\begin{lem}[{\cite[Lemma 3.3]{andrada-origlia}}]\label{Vaisman}
Let $\mathfrak g$ be a Lie algebra with an lcK structure $(J, \Omega, \theta)$ and let $A \in \mathfrak g$ be such that $A \in (\ker \theta)^\perp$ with respect to the compatible metric given by $(J, \Omega)$ and $\theta (A)=1$. Then $(J, \Omega, \theta)$ is Vaisman if and only if the adjoint operator $\operatorname{ad}_A$ is a skew-symmetric endomorphism of $\mathfrak g$.
\end{lem}

Using this characterization, we determine, among all lcK structures for each $4$-dimensional Lie algebra, the ones being Vaisman (see also \cite{andrada-origlia-Vaisman} for a nice description of unimodular Lie algebras admitting Vaisman structures).

With the same aim as \cite[page 56]{ovando}, hopefully the classification in Table \ref{table:lck-solvable} might be useful in the future to provide specific examples and to solve open problems. In this direction we use our classification to exhibit explicit examples of Lie algebras in dimension higher than four admitting an lcK structure.
More precisely, we adapt some constructions in \cite{origlia-construction, angella-bazzoni-parton} to the lcK case and, as an application, we can produce many examples in higher dimension starting from dimension four, as well as we give a geometric interpretation of some of the $4$-dimensional structures in Table \ref{table:lck-solvable}.
In particular, the lcK extension discussed in Proposition \ref{construccionLCK} allows to recover lcK structures on Oeljeklaus-Toma manifolds \cite{oeljeklaus-toma}.

\bigskip

\begin{small}
\noindent{\itshape Notation.}
Structure equations for Lie algebras are written using the Salamon notation: {\itshape e.g.} $\mathfrak{rh}_3=(0,0,-12,0)$ means that we fix a coframe $(e^1,e^2,e^3,e^4)$ for $\mathfrak{rh}_3^\vee$ such that $de^1=de^2=de^4=0$ and $de^3=-e^1\wedge e^2$. Complex structures and tensors are usually expressed in terms of the above coframe. For example, complex structures on $\mathfrak g$ are defined in terms of their dual $J \colon \mathfrak{g}^\vee \to \mathfrak{g}^\vee$ with the convention $J\alpha=\alpha(J^{-1}\_)$.
By lcs structure, we mean a non-symplectic structure, namely, the Lee form is assumed to be non-exact (actually non-zero).
\end{small}

\bigskip

\begin{small}
\noindent{\itshape Acknowledgement.}
Part of this work has been set up during the visit of the first-named author at Universidad Nacional de C\'ordoba with the support of the Istituto Italiano di Cultura: he warmly thanks prof. Adri\'an Andrada, the Facultad de Matem\'atica, Astronom\'ia, F\'isica y Computaci\'on, and the Istituto Italiano di Cultura in C\'ordoba for their warm hospitality.
The authors would like to thank Adri\'an Andrada and Giovanni Bazzoni for several interesting discussions on the subject, and the anonymous Referee for her/his comments that improved the presentation of the paper.
\end{small}

\section{Classification of lcK structures on four-dimensional Lie algebras}

In this section, we summarize the classification of locally conformally K\"ahler structures on $4$-dimensional Lie algebras up to linear equivalence. Here, by {\em linear equivalent} lcK structures $(J_1, \Omega_1, \theta_1)$ and $(J_2, \omega_2, \theta_2)$ on the Lie algebra $\mathfrak g$ of dimension $\dim_{\mathbb R} \mathfrak g \geq 4$ we mean that there is an automorphism $A\in\mathfrak{gl}(\mathfrak{g})$ of the Lie algebra such that $J_2 = A^{-1} \circ J_1 \circ A$ and $\Omega_1 = A^*\Omega_2 = \Omega_2(A\_, A\_)$; by the injectivity of $\Omega_1\wedge\_ \colon \wedge^1\mathfrak{g}^\vee \to \wedge^3\mathfrak{g}^\vee$, we also get $\theta_1=A^*\theta_2$.

\subsection{LcK structures on four-dimensional solvable Lie algebras}
Complex structures on $4$-dimensional solvable Lie algebras up to linear equivalence are classified by G. Ovando, see \cite{ovando} and references therein. In Table \ref{table:complex}, we summarize her results as in \cite[Proposition 3.2]{ovando} (note just the correction of a typo in case $J_2$ for $\mathfrak{rr}^\prime_{3,\gamma}$). Locally conformally symplectic structures on $4$-dimensional Lie algebras are classified in \cite{angella-bazzoni-parton}.
In the next Section \ref{sec:proof-class-lck}, we will combine these classification results to get the following.

\begin{table}[h]
\def\arraystretch{1.5}
\centering
{\resizebox{\textwidth}{!}{
\begin{tabular}{lcc|ll|l|c}
   Lie algebra & c.s. & n. & parameters & structure equations & complex structure & $\mathcal{Z}(\mathfrak g)$  \\
\toprule
    \hline
    $\mathfrak r\mathfrak h_3$ & \checkmark & \checkmark & & $(0,0,-12,0)$ & $Je^1=e^2$, $Je^3=e^4$ & $\langle e_3, e_4\rangle$\\
    \hline
    \multirow{2}{*}{$\mathfrak r\mathfrak r_{3,\lambda}$} & \multirow{2}{*}{\checkmark} & \multirow{2}{*}{$\times$} & $\lambda=0$ & $(0,-12,0,0)$ & $Je^1=e^2$, $Je^3=e^4$ & \multirow{2}{*}{$\langle e_4\rangle$}\\\cline{4-6}
    & & & $\lambda=1$ & $(0,-12,-13,0)$ & $Je^1=e^4$, $Je^2=-e^3$ & \\
    \hline
    \multirow{3}{*}{$\mathfrak r\mathfrak r'_{3,\gamma}$} & \multirow{3}{*}{$\times$} & \multirow{3}{*}{$\times$} & $\gamma=0$ & $(0,-13,12,0)$ & $Je^1=e^4$, $Je^2=e^3$ & \multirow{3}{*}{$\langle e_4\rangle$} \\\cline{4-6}
    & & & \multirow{2}{*}{$\gamma>0$} & \multirow{2}{*}{$(0,-\gamma 12-13,12-\gamma 13,0)$} & $J_1e^1=e^4$, $J_1e^2=e^3$ & \\\cline{6-6}
    & & & & & $J_2e^1=e^4$, $J_2e^2=-e^3$ & \\
    \hline
    $\mathfrak r_2\mathfrak r_2$ & \checkmark & $\times$ & & $(0,-12,0,-34)$ & $Je^1=e^2$, $Je^3=e^4$ & $0$\\
    \hline
    \multirow{2}{*}{$\mathfrak r'_2$}  & \multirow{2}{*}{$\times$} & \multirow{2}{*}{$\times$} & & \multirow{2}{*}{$(0,0,-13+24,-14-23)$} & $J_1e^1=e^3$, $J_1e^2=e^4$ & \multirow{3}{*}{$0$}\\\cline{6-6}
    & & & & & $J_2e^1=-ae^1+\frac{a^2+1}{b} e^2$, $J_2e^3=e^4$ &\\
    & & & & & with $b\neq0$ &\\
    \hline
    $\mathfrak r_{4,\mu}$ & \checkmark & $\times$ & $\mu=1$ & $(14,24+34,34,0)$ & $Je^1=e^2$, $Je^3=-e^4$ & $0$ \\
    \hline
    \multirow{2}{*}{$\mathfrak r_{4,\alpha,\beta}$} & \multirow{2}{*}{\checkmark} & \multirow{2}{*}{$\times$} & $-1<\alpha<1$, $\alpha\neq 0$, $\beta=1$ & $(14,\alpha 24,34,0)$ & $Je^1=e^3$, $Je^2=-e^4$ & \multirow{2}{*}{$0$}\\\cline{4-6}
    & & & $-1\leq\alpha=\beta< 1$, $\alpha\neq0$ & $(14,\alpha 24,\alpha 34,0)$ & $Je^1=-e^4$, $Je^2=e^3$ \\
    \hline
    \multirow{2}{*}{$\mathfrak r'_{4,\gamma,\delta}$} & \multirow{2}{*}{$\times$} & \multirow{2}{*}{$\times$} & \multirow{2}{*}{$\gamma\in\mathbb R$, $\delta>0$} & \multirow{2}{*}{$(14,\gamma 24+\delta 34,-\delta 24+\gamma 34,0)$} & $J_1e^1=-e^4$, $J_1e^2=e^3$ & \multirow{2}{*}{$0$}\\\cline{6-6}
    & & & & & $J_2e^1=-e^4$, $J_2e^2=-e^3$ & \\
    \hline
    \multirow{2}{*}{$\mathfrak d_4$} & \multirow{2}{*}{\checkmark} & \multirow{2}{*}{$\times$} & & \multirow{2}{*}{$(14,-24,-12,0)$} & $J_1e^1=-e^3$, $J_1e^2=-e^4$ & \multirow{2}{*}{$\langle e_3\rangle$}\\\cline{6-6}
    & & & & & $J_2e^1=e^2-e^3$, $J_2e^2=-e^4$ & \\
    \hline
    \multirow{6}{*}{$\mathfrak d_{4,\lambda}$} & \multirow{6}{*}{\checkmark} & \multirow{6}{*}{$\times$} & $\lambda=1$ & $(14,0,-12+34,0)$ & $Je^1=e^4$, $Je^2=e^3$ & \multirow{6}{*}{$0$}\\\cline{4-6}
    & & & \multirow{3}{*}{$\lambda=\frac{1}{2}$} & \multirow{3}{*}{$(\frac12\times 14,\frac12\times 24,-12+34,0)$} & $J_1e^1=e^2$, $J_1e^3=-e^4$ & \\\cline{6-6}
    & & & & & $J_2e^1=-e^2$, $J_2e^3=-e^4$ & \\\cline{6-6}
    & & & & & $J_3e^1=-e^4$, $J_3e^2=-\frac12 e^3$ & \\\cline{4-6}
    & & & \multirow{2}{*}{$\lambda>\frac{1}{2}$, $\lambda\neq1$} & \multirow{2}{*}{$(\lambda 14,(1-\lambda)24,-12+34,0)$} & $J_1e^1=\lambda e^4$, $J_1e^2=e^3$ & \\\cline{6-6}
    & & & & & $J_2e^1=e^3$, $J_2e^2=(\lambda-1)e^4$ & \\
    \hline
    \multirow{6}{*}{$\mathfrak d'_{4,\delta}$} & \multirow{6}{*}{$\times$} & \multirow{6}{*}{$\times$} & \multirow{2}{*}{$\delta=0$} & \multirow{2}{*}{$(24,-14,-12,0)$} & $J_2e^1=e^2$, $J_2e^3=e^4$ & \multirow{2}{*}{$\langle e_3\rangle$}\\\cline{6-6}
    & & & & & $J_3e^1=e^2$, $J_3e^3=-e^4$ & \\\cline{4-7}
    & & & \multirow{4}{*}{$\delta>0$} & \multirow{4}{*}{$(\frac{\delta}{2}14+24,-14+\frac{\delta}{2}24,-12+\delta 34,0)$} & $J_1e^1=-e^2$, $J_1e^3=-e^4$ & \multirow{4}{*}{$0$}\\\cline{6-6}
    & & & & & $J_2e^1=e^2$, $J_2e^3=e^4$ & \\\cline{6-6}
    & & & & & $J_3e^1=e^2$, $J_3e^3=-e^4$ & \\\cline{6-6}
    & & & & & $J_4e^1=-e^2$, $J_4e^3=e^4$ & \\\hline
    $\mathfrak h_{4}$ & \checkmark & $\times$ & & $(\frac{1}{2}14+24,\frac{1}{2}24,-12+34,0)$ & $Je^1=\frac12 e^3$, $Je^2=-e^4$ & $0$ \\
    \hline
    \bottomrule
\end{tabular}
}}
\caption{Complex non-K\"ahler structures on solvable Lie algebras in dimension $4$ up to linear equivalence, following \cite{ovando}. ("c.s." is for "completely-solvable", "n." is for "nilpotent", and $\mathcal{Z}$ denotes the center.)
}
\label{table:complex}
\end{table}

\begin{thm}\label{thm:classification-lck}
Non-K\"ahler locally conformally K\"ahler structures on $4$-dimensional solvable Lie algebras are classified up to linear equivalence in Table \ref{table:lck-solvable} in Appendix \ref{app:table}.
\end{thm}

Existence of lattices for solvable Lie groups is investigated in \cite{bock}, and compact complex surfaces diffeomorphic to solvmanifolds are studied in \cite{hasegawa-jsg}.
Recall that:
\begin{itemize}
\item $\mathbb R^4$ is the Lie algebra associated to complex tori;
\item $\mathfrak{rh}_3$ is the Lie algebra associated to primary Kodaira surfaces;
\item $\mathfrak{rr}^\prime_{3,0}$ is the Lie algebra associated to hyperelliptic surfaces;
\item $\mathfrak{r}^\prime_{4,-\frac{1}{2},\delta}$ with $\delta>0$ is the Lie algebra associated to Inoue surfaces of type $S^0$;
\item $\mathfrak{d}_4$ is the Lie algebra associated to Inoue surfaces of type $S^+$;
\item $\mathfrak{d}^\prime_{4,0}$ is the Lie algebra associated to secondary Kodaira surfaces;
\item and the Lie groups associated to the other algebras do not admit compact quotients.
\end{itemize}
See also \cite{belgun, pontecorvo-vaisman} and references therein as for the problem of existence of lcK structures on compact complex surfaces, known as the {\em Vaisman question}.

\subsection{LcK structures on four-dimensional reductive Lie algebras}
Locally conformally K\"ahler structures on $4$-dimensional reductive Lie algebras are studied in \cite[Theorem 4.6]{achk}. More precisely, they show that the only $4$-dimensional reductive Lie algebras admitting locally conformally pseudo-K\"ahler structures are $\mathfrak{gl}_2=\mathfrak{sl}_2\oplus\mathbb R=(-23, -2\times 12, 2\times13,0)$ (see \cite[Proposition 4.7]{achk} for complex structures, see \cite[Section 3.2]{angella-bazzoni-parton} for lcs structures) and $\mathfrak{u}_2=\mathfrak{su}_2\oplus\mathbb{R}=(23, -13, 12, 0)$ (see \cite[Proposition 4.4]{achk} for complex structures, see \cite[Section 3.1]{angella-bazzoni-parton} for lcs structures).
For the sake of completeness, we recall here the lcK structures in \cite[Theorem 4.9, Theorem 4.6]{achk}.

\subsubsection{$\mathfrak{gl}_2$}
Consider the Lie algebra
$$ \mathfrak{gl}_2 = \mathfrak{sl}_2\oplus\mathbb R = (-23, -2\times 12, 2\times13,0) , $$
where $e_4$ is a generator of $\mathbb R$.
Left-invariant complex structures are described in \cite[Proposition 4.7]{achk}: they belong to two families, both depending on one parameter $\mu = \mu_1+\sqrt{-1}\mu_2 \in \mathbb C \setminus \sqrt{-1}\mathbb R$, and they are defined by
$$ \left\{\begin{array}{rcl}
J_{1,\mu} e_1 &=& e_2+e_3 \\
J_{1,\mu} e_2 &=& - \frac{1}{2} e_1 - \frac{\mu_2}{2\mu_1} (e_2-e_3) + \frac{1}{\mu_1}e_4 \\
J_{1,\mu} e_3 &=& - \frac{1}{2} e_1 + \frac{\mu_2}{2\mu_1} (e_2-e_3) - \frac{1}{\mu_1}e_4 \\
J_{1,\mu} e_4 &=& - \frac{|\mu|^2}{2\mu_1}(e_2-e_3) + \frac{\mu_2}{\mu_1} e_4
\end{array}\right.
$$
and
$$ \left\{\begin{array}{rcl}
J_{2,\mu} e_1 &=& -e_2-e_3 \\
J_{2,\mu} e_2 &=& \frac{1}{2} e_1 - \frac{\mu_2}{2\mu_1} (e_2-e_3) - \frac{1}{\mu_1}e_4 \\
J_{2,\mu} e_3 &=& \frac{1}{2} e_1 + \frac{\mu_2}{2\mu_1} (e_2-e_3) + \frac{1}{\mu_1}e_4 \\
J_{2,\mu} e_4 &=& \frac{|\mu|^2}{2\mu_1}(e_2-e_3) + \frac{\mu_2}{\mu_1} e_4
\end{array}\right.
$$
namely, with respect to the dual coframe, they are associated to the matrices
$$ J_{1,\mu} = \left( \begin{matrix}
0 & -1 & -1 & 0 \\
\frac{1}{2} & \frac{\mu_{2}}{2 \, \mu_{1}} & -\frac{\mu_{2}}{2 \, \mu_{1}} & -\frac{1}{\mu_{1}} \\
\frac{1}{2} & -\frac{\mu_{2}}{2 \, \mu_{1}} & \frac{\mu_{2}}{2 \, \mu_{1}} & \frac{1}{\mu_{1}} \\
0 & \frac{\mu_{1}^{2} + \mu_{2}^{2}}{2 \, \mu_{1}} & -\frac{\mu_{1}^{2} + \mu_{2}^{2}}{2 \, \mu_{1}} & -\frac{\mu_{2}}{\mu_{1}}
\end{matrix}\right) \in \mathrm{End}(\mathfrak{gl}_2^\vee) $$
and
$$
J_{2,\mu} = \left( \begin{matrix}
0 & 1 & 1 & 0 \\
-\frac{1}{2} & \frac{\mu_{2}}{2 \, \mu_{1}} & -\frac{\mu_{2}}{2 \, \mu_{1}} & \frac{1}{\mu_{1}} \\
-\frac{1}{2} & -\frac{\mu_{2}}{2 \, \mu_{1}} & \frac{\mu_{2}}{2 \, \mu_{1}} & -\frac{1}{\mu_{1}} \\
0 & -\frac{\mu_{1}^{2} + \mu_{2}^{2}}{2 \, \mu_{1}} & \frac{\mu_{1}^{2} + \mu_{2}^{2}}{2 \, \mu_{1}} & -\frac{\mu_{2}}{\mu_{1}}
\end{matrix}\right) \in \mathrm{End}(\mathfrak{gl}_2^\vee)
. $$
These two families are related by an automorphism of $\mathfrak{gl}_2$, namely,
$$
\left(\begin{array}{cccc}
1 & 0 & 0 & 0 \\
0 & -1 & 0 & 0 \\
0 & 0 & -1 & 0 \\
0 & 0 & 0 & 1
\end{array}\right) \in \mathrm{Aut}(\mathfrak{gl}_2) .
$$
Then it is sufficient to consider the family $J_{1,\mu}$.
The only non-trivial automorphism of $(\mathfrak{gl}_2,J_{1,\mu})$ is
\begin{equation}\label{aut_gl2}
\left(\begin{array}{cccc}
-1 & 0 & 0 & 0 \\
0 & 0 & -1 & 0 \\
0 & -1 & 0 & 0 \\
0 & 0 & 0 & 1
\end{array}\right) \in \mathrm{Aut}(\mathfrak{gl}_2) .
\end{equation}
The generic lcs structure is
$$ \theta = \theta_{4}  e^{4},
\quad
\Omega = \omega_{12}  e^{12} + \omega_{13}  e^{13} - \omega_{23} \theta_{4}  e^{14} + \omega_{23}  e^{23} - \frac{1}{2} \, \omega_{12} \theta_{4}  e^{24} + \frac{1}{2} \, \omega_{13} \theta_{4}  e^{34} , $$
with $( \omega_{12} \omega_{13} - \omega_{23}^{2} ) \theta_{4} \neq 0$.

Assuming $\Omega$ is $J_{1,\mu}$-invariant we have two cases. Indeed, the condition reduces to the system
$$ \left(\begin{matrix}
\mu_1+\theta_4 & -\mu_1-\theta_4 & -2\mu_2 \\
\mu_2& -\mu_2 & -2(\mu_1^2+\mu_2^2)-2\mu_1\theta_4 \\
\mu_2&-\mu_2&2\mu_1+2\theta_4 \\
(\mu_1^2+\mu_2^2)+\mu_1\theta_4 & -(\mu_1^2+\mu_2^2)-\mu_1\theta_4&-2\mu_2\theta_4\end{matrix}\right)
\cdot \left(\begin{matrix}\omega_{12}\\\omega_{13}\\\omega_{23}\end{matrix}\right)
= \left(\begin{matrix}0\\0\\0\\0\end{matrix}\right).$$

We consider first when $\mu_{2}=0$ and $\theta_{4}= -\mu_{1}$, when the rank of the above $4\times3$ matrix is zero.
In this case, $\Omega$ is always $J_{1,\mu}$-invariant.
The $J_{1,\mu}$-positivity of the lcs structure forces $\omega_{12} > 0$, $\omega_{13} > 0$, and $\omega_{12}\omega_{13}-\omega_{23}^2 > 0$. And the general lck structure is 
$$\left\{
\begin{array}{l}
\theta = -\mu_{1}  e^{4} \\
\Omega = \omega_{12}  e^{12} + \omega_{13}  e^{13} + \omega_{23} \mu_{1}  e^{14} + \omega_{23}  e^{23} + \frac{1}{2} \, \omega_{12} \mu_{1}  e^{24} - \frac{1}{2} \, \omega_{13} \mu_{1}  e^{34} \\
\text{with } \omega_{12}>0, \omega_{13} > 0, \omega_{12}\omega_{13}-\omega_{23}^2 > 0
\end{array}\right..$$
Applying the only non-trivial automorphism we can assume that $ \omega_{12}\geq \omega_{13}$ and  $\omega_{23}\geq0$:
\begin{equation}\label{eq:gl2-A}
\left\{
\begin{array}{l}
\theta = -\mu_{1}  e^{4} \\
\Omega = \omega_{12}  e^{12} + \omega_{13}  e^{13} + \omega_{23} \mu_{1}  e^{14} + \omega_{23}  e^{23} + \frac{1}{2} \, \omega_{12} \mu_{1}  e^{24} - \frac{1}{2} \, \omega_{13} \mu_{1}  e^{34} \\
\text{with } \omega_{12}\geq\omega_{13} > 0, \omega_{23}\geq0, \omega_{12}\omega_{13}-\omega_{23}^2 > 0
\end{array}\right.,
\end{equation}
see also \cite[Theorem 4.9, item (ii)]{achk}.

In the other case, when $\mu_{2}\neq0$ or $\theta_{4}\neq -\mu_{1}$, the rank of the above $4\times3$ matrix is two.
The compatibility of the lcs with the complex structure $J_{1,\mu}$ forces $\omega_{23}=0$ and $\omega_{12}=\omega_{13}$, with $\omega_{12} > 0$, $\frac{\theta_4}{\mu_1} < 0$. 
Summarizing, up to equivalence, the lcK structure on $(\mathfrak{gl}_2,J_{1,\mu})$ is of the form, see \cite[Theorem 4.9, item (i)]{achk},
\begin{equation}\label{eq:gl2-B}
\left\{
\begin{array}{l}
\theta = \theta_{4}  e^{4} \\
\Omega = \omega_{12}  e^{12} + \omega_{12}  e^{13} - \frac{1}{2} \, \omega_{12} \theta_{4}  e^{24} + \frac{1}{2} \, \omega_{12} \theta_{4}  e^{34} \\
\text{with } \omega_{12}>0, \frac{\theta_4}{\mu_1} < 0 , \mu_2^2+(\theta_4+\mu_1)^2\neq0
\end{array}\right..
\end{equation}
And there is no further reduction since the non-trivial automorphism fixes the lck structure.

\begin{rmk}
Among the above lcK structures, the ones being Vaisman are specified in \cite[Theorem 4.9]{achk}. More precisely, lcK structures of type \eqref{eq:gl2-A} are Vaisman if and only if $\omega_{23}=0$ and $\omega_{12}=\omega_{13}$; and lcK structures of type \eqref{eq:gl2-B} are always Vaisman. In other words, Vaisman structures are
$$
\left\{
\begin{array}{l}
\theta = \theta_{4}  e^{4} \\
\Omega = \omega_{12}  e^{12} + \omega_{12}  e^{13} - \frac{1}{2} \, \omega_{12} \theta_{4}  e^{24} + \frac{1}{2} \, \omega_{12} \theta_{4}  e^{34} \\
\text{with } \omega_{12}>0, \frac{\theta_4}{\mu_1} < 0
\end{array}\right..
$$
\end{rmk}

\subsubsection{$\mathfrak u_2$}
Consider the Lie algebra
$$ \mathfrak{u}_2 = \mathfrak{su}_2\oplus\mathbb R = (23, -13, 12, 0) , $$
where $e_4$ is a generator of $\mathbb R$.
Left-invariant complex structures are described in \cite[Proposition 4.4]{achk}: they depend on two-parameters $a \in \mathbb R$ and $b \in \mathbb R \setminus\{0\}$ and they are defined by
$$ J_{a,b} e_4 = b e_1 + a e_4, \qquad J_{a,b} e_3 = e_2 , $$
namely, with respect to the dual coframe, it is associated to the matrix
$$ J_{a,b} = \left( \begin{matrix}
a & 0 & 0 & \frac{a^{2} + 1}{b} \\
0 & 0 & 1 & 0 \\
0 & -1 & 0 & 0 \\
-b & 0 & 0 & -a
\end{matrix}\right) \in \mathrm{End}(\mathfrak{u}_2^\vee) . $$
The automorphisms of $(\mathfrak{u}_2,J_{a,b})$ are of the form
$$
\left(\begin{array}{cccc}
1 & 0 & 0 & 0 \\
0 & a_{22} & a_{23} & 0 \\
0 & -a_{23} & a_{22} & 0 \\
0 & 0 & 0 & 1
\end{array}\right) \in \mathrm{Aut}(\mathfrak{u}_2)
$$
with the condition $a_{22}^2 + a_{23}^2 = 1$.

The generic lcs structure is
$$ \theta=\theta_{4}  e^{4}, \qquad
\Omega=\omega_{12}  e^{12} + \omega_{13}  e^{13} + \omega_{23} \theta_{4}  e^{14} + \omega_{23}  e^{23} - \omega_{13} \theta_{4}  e^{24} + \omega_{12} \theta_{4}  e^{34} $$
with $\left( \omega_{12}^{2} + \omega_{13}^{2} + \omega_{23}^{2} \right)\theta_{4} \neq0$.

By imposing the $J_{a,b}$-invariance, we get the condition
$$ \left(\begin{matrix}
b+\theta_4 & a\theta_{4} \\\theta_{4}a & -(b+\theta_4) \end{matrix}\right) \cdot \left(\begin{matrix}
\omega_{12} \\ \omega_{13} \end{matrix} \right) =
\left(\begin{matrix} 0 \\ 0 \end{matrix}\right) . $$

Assuming $\Omega$ to be $J_{a,b}$-positive we obtain $\omega_{23}<0, \frac{\theta_4}{b} > 0$. In particular $b\neq -\theta_{4}$, and therefore the $J_{a,b}$-invariance implies that $\omega_{12}=\omega_{13}=0$. Summarizing we reduce to the generic lcK structure, see \cite[Theorem 4.6, item (i)]{achk},
\begin{equation}\label{eq:lck-u2}
\left\{ \begin{array}{l}
\theta=\theta_4 e^4 \\
\Omega=\omega_{23}\theta_4 e^{14}+\omega_{23}e^{23} \\
\text{with } \theta_4 \not\in \{0, -\frac{1}{b}\}, \omega_{23}<0, \frac{\theta_4}{b} > 0
\end{array}\right.,
\end{equation}
and no further reduction is possible since a possible automorphism fixs the lcK structure.

\begin{rmk}
Among the above lcK structures, the ones being Vaisman are specified in \cite[Theorem 4.6]{achk}. More precisely, all the lcK structures in \eqref{eq:lck-u2} are of Vaisman type.
\end{rmk}

\section{Proof of Theorem \ref{thm:classification-lck}: lcK structures on four-dimensional solvable Lie algebras}\label{sec:proof-class-lck}
	
\subsection{$\mathfrak{rh}_{3}$}
Consider the Lie algebra $\mathfrak{rh}_3=(0,0,-12,0)$, namely, $(e^1,e^2,e^3,e^4)$ is a coframe of $1$-forms such that
$$ de^1 = 0, \quad de^2 = 0, \quad de^3 = -e^1\wedge e^2, \quad de^4=0 . $$
Equivalently, in terms of the dual frame $(e_1, e_2, e_3, e_4)$ for $\mathfrak{rh}_3$, we have the structure equations
$$ [e_1,e_2]=e_3, \quad [e_1,e_3]=0, \quad [e_1,e_4]=0, $$
$$ [e_2,e_3]=0, \quad [e_2,e_4]=0, \quad [e_3,e_4]=0. $$

According to \cite{ovando}, there is only one complex structure up to linear equivalence.
In terms of the frame for $\mathfrak{rh}_3$, it is given by specifying the $(-\sqrt{-1})$-eigenspace to be $\langle e_1 + \sqrt{-1}\,e_2, e_3 + \sqrt{-1}\, e_4\rangle$; namely, $Je_1=e_2$, $Je_2=-e_1$, $Je_3=e_4$, $Je_4=-e_3$. On $\mathfrak{rh}_3^\vee$, we set the linear complex structure $J\in\mathrm{End}(\mathfrak{rh}_3^\vee)$ by $J\alpha:=\alpha(J^{-1}\_)$. Then, in terms of the coframe above, we have $Je^1=e^2$, $Je^2=-e^1$, $Je^3=e^4$, $Je^4=-e^3$, that is, $J$ is given by the matrix
$$
J=\left(\begin{array}{cccc}
0 & -1 & 0 & 0 \\
1 & 0 & 0 & 0 \\
0 & 0 & 0 & -1 \\
0 & 0 & 1 & 0
\end{array}\right) \in\mathrm{End}(\mathfrak{rh}_3^\vee) .
$$

As in \cite[Appendix 6.1 of the arXiv version, page 28]{angella-bazzoni-parton}, by requiring $d\theta=0$, $d\Omega-\theta\wedge\Omega=0$, and $\Omega\wedge\Omega\neq0$, we get that the generic (non-symplectic) lcs structure is of the form
$$ \left\{\begin{array}{l}
\theta = \theta_1 e^1+ \theta_2 e^2 + \theta_4 e^4 \\
\Omega = \left( -\frac{\omega_{24} \theta_{1} - \omega_{14} \theta_{2} + \omega_{34}}{\theta_{4}} \right)  e^{12} - \frac{\omega_{34} \theta_{1}}{\theta_{4}}  e^{13} + \omega_{14}  e^{14} - \frac{\omega_{34} \theta_{2}}{\theta_{4}}  e^{23} + \omega_{24}  e^{24} + \omega_{34}  e^{34} \\
\text{with } \theta_4\neq0, \omega_{34}\neq0
\end{array}\right. . $$

We impose now $\Omega$ to be $J$-invariant, (namely, $J\omega:=\omega(J\_,J\_)=\omega$,) and $J$-positive, (namely, $\omega(x,Jx)>0$ for any $x\in\mathfrak{rh}_3\setminus\{0\}$).
The $J$-invariance forces $\omega_{24} = -\frac{\omega_{34}\theta_1}{\theta_4}$ and $\omega_{14} = \frac{\omega_{34}\theta_2}{\theta_4}$, namely, we get
$$ \Omega = \left( \frac{\theta_{1}^{2} + \theta_{2}^{2} - \theta_{4}}{\theta_{4}^{2}} \right)\omega_{34}   e^{12} - \frac{\omega_{34} \theta_{1}}{\theta_{4}}  e^{13} + \frac{\omega_{34} \theta_{2}}{\theta_{4}}  e^{14} - \frac{\omega_{34} \theta_{2}}{\theta_{4}}  e^{23} - \frac{\omega_{34} \theta_{1}}{\theta_{4}}  e^{24} + \omega_{34}  e^{34} .$$
In our case, for the $J$-positivity, it suffices to check that $\omega(e_1,Je_1)=\left( \frac{\theta_{1}^{2} + \theta_{2}^{2} - \theta_{4}}{\theta_{4}^{2}} \right)\omega_{34}>0$ and $\omega(e_3,Je_3)=\omega_{34}>0$ and $\frac{1}{2}\omega^2(e_1,Je_1,e_3,Je_3)=-\frac{\omega_{34}^2}{\theta_4}>0$.) We get that the generic lcK structure is of the form
$$ \left\{\begin{array}{l}
\theta = \theta_1 e^1+ \theta_2 e^2 + \theta_4 e^4 \\
\Omega = \left( \frac{\theta_{1}^{2} + \theta_{2}^{2} - \theta_{4}}{\theta_{4}^{2}} \right)\omega_{34}   e^{12} - \frac{\omega_{34} \theta_{1}}{\theta_{4}}  e^{13} + \frac{\omega_{34} \theta_{2}}{\theta_{4}}  e^{14} - \frac{\omega_{34} \theta_{2}}{\theta_{4}}  e^{23} - \frac{\omega_{34} \theta_{1}}{\theta_{4}}  e^{24} + \omega_{34}  e^{34} \\
\text{with }	\omega_{34}>0, \theta_4<0
\end{array}\right. . $$

The generic complex automorphism of $(\mathfrak{rh}_3^\vee,J)$ are given, with respect to the chosen coframe, by
$$
\left(\begin{array}{cccc}
a_{11} & a_{12} & a_{13} & a_{14} \\
-a_{12} & a_{11} & -a_{14} & a_{13} \\
0 & 0 & a_{11}^{2} + a_{12}^{2} & 0 \\
0 & 0 & 0 & a_{11}^{2} + a_{12}^{2}
\end{array}\right) \in \mathrm{Aut}(\mathfrak{h}_3^\vee)
$$
with the condition
$$ a_{11}^{2} + a_{12}^{2} \neq 0 . $$

First, we apply the automorphism with parameters $a_{11}=0$, $a_{12}=1$, $a_{13}=\frac{\theta_{1}}{\theta_{4}}$, and $a_{14}=-\frac{\theta_{2}}{\theta_{4}}$. This reduces the lcK structure to $\theta=\theta_{4}  e^{4}$ and $\Omega=-\frac{\omega_{34}}{\theta_{4}}  e^{12} + \omega_{34}  e^{34}$, where $\omega_{34}<0$ and $\theta_4<0$.
Then we apply the automorphism with parameters $a_{12}=\sqrt{-\frac{1}{\theta_{4}}}$, the others zero, so to transform the generic lcK form in (we set $\sigma=\frac{\omega_{34}}{\theta_{4}^{2}}$)
$$ \left\{\begin{array}{l}
\theta = -e^{4} \\
\Omega = \sigma e^{12} + \sigma e^{34} \\
\text{with }	\sigma > 0
\end{array}\right. . $$
It is easy to see that such forms cannot be further reduced, since the generic automorphism transforms $\theta$ as $-a_{14}e^1 - a_{13}e^2 + (-a_{11}^2 - a_{12}^2)e^4$, and correspondingly the coefficient of $\Omega$ along $e^{12}$ as $\left( {\left(a_{11}^{2} + a_{12}^{2} + a_{13}^{2} + a_{14}^{2}\right)} \sigma \right)$.

\begin{rmk}
We determine now which of these lcK structures on $\mathfrak{rh}_{3}$ are of Vaisman type. Let $A=a_1e_1+a_2e_2+a_3e_3+a_4e_4$. We determine $a_i$ such that $\theta(A)=1$ and $A \in (\ker \theta)^\perp$, that is, $\Omega(A,Jx)=0$ for any $x\in \ker \theta$. In this case $\ker \theta$ is generated by $\{e_1,e_2,e_3\}$ and we obtain that 
	$A=-e_4\in \mathcal{Z}(\mathfrak g)$ and $\operatorname{ad}_A=0$. Therefore, it follows from Lemma \ref{Vaisman} that all the lcK structures above are of Vaisman type.
\end{rmk}

\begin{rmk}
We observe that the Morse-Novikov cohomology with respect to $\theta=-e^4$ vanishes in any degree.
\end{rmk}

\subsection{$\mathfrak{rr}_{3,0}$}
Consider the Lie algebra $\mathfrak{rr}_{3,0} = (0,-12, 0, 0)$ with the complex structure defined as
$$
J = \left(\begin{array}{cccc}
0 & -1 & 0 & 0 \\
1 & 0 & 0 & 0 \\
0 & 0 & 0 & -1 \\
0 & 0 & 1 & 0
\end{array}\right) \in \mathrm{End}(\mathfrak{rr}^\vee_{3,0})
$$
in terms of the chosen coframe.

According to \cite[Appendix 6.3 of the arXiv version, pages 31--32]{angella-bazzoni-parton}, the generic lcs structures fall in two different families.

We first consider the case when the generic closed $1$-form $\theta=\theta_1 e^1+\theta_3e^3+\theta_4e^4$ has $\theta_3=\theta_4=0$. Then the generic lcs structure with Lee form $\theta$ is
$$ \left\{\begin{array}{l}
\theta = -e^1 \\
\Omega = \omega_{12} e^{12} + \omega_{13} e^{13} + \omega_{14} e^{14} + \omega_{23} e^{23} + \omega_{24} e^{24} \\
\text{with } \omega_{14}\omega_{23}-\omega_{13}\omega_{24}\neq0
\end{array}\right. . $$
It is clear that such a form is never $J$-positive: indeed, with respect to the dual frame $(e_1,e_2,e_3,e_4)$, we have $\omega(e_3, Je_3)=0$. Then, there is no lcK structure in this case.

Consider now the case $\theta_3^2 + \theta_4^2 \neq 0$. The generic complex automorphisms $(\mathfrak{rr}_{3,0}, J)$ are given, with respect to the chosen coframe, by
\[\left(\begin{array}{cccc}
1 & 0 & 0 & 0 \\
0 & 1 & 0 & 0 \\
0 & 0 & a_{33} & a_{34} \\
0 & 0 & -a_{34} & a_{33}
\end{array}\right) \in \mathrm{Aut}(\mathfrak{rr}_{3,0}^\vee) , \qquad \text{ with } a_{33}^{2} + a_{34}^{2}\neq 0.\]
In particular, the complex automorphism
$$
\left(\begin{array}{cccc}
1 & 0 & 0 & 0 \\
0 & 1 & 0 & 0 \\
0 & 0 & 0 & 1 \\
0 & 0 & -1 & 0
\end{array}\right) \in \mathrm{Aut}(\mathfrak{rr}_{3,0}^{\vee})
$$
transforms $\theta_{1}  e^{1} + \theta_{3}  e^{3} + \theta_{4}  e^{4}$ to $\theta_{1}  e^{1} + \theta_{4}  e^{3} - \theta_{3}  e^{4}$. So, without loss of generality, we can assume $\vartheta_3\neq0$. The generic lcs structure in this case reduces to
$$ \left\{\begin{array}{l}
\theta = \theta_1e^1+\theta_3e^3+\theta_4e^4 \\
\Omega = \left( -\frac{(\theta_1+1)\omega_{23}}{\theta_{3}} \right)  e^{12} + \omega_{13}  e^{13} + \left( \frac{\omega_{34} \theta_{1} + \omega_{13} \theta_{4}}{\theta_{3}} \right)  e^{14} + \omega_{23}  e^{23} + \frac{\omega_{23} \theta_{4}}{\theta_{3}}  e^{24} + \omega_{34}  e^{34} \\
\text{with } \theta_3\neq 0, \omega_{23} \omega_{34}\neq0
\end{array}\right. . $$
The $J$-invariance requires $\omega_{13} = -\frac{\omega_{34} \theta_{1} \theta_{4}}{\theta_{3}^{2} + \theta_{4}^{2}}$ and $\omega_{23} = -\frac{\omega_{34} \theta_{1} \theta_{3}}{\theta_{3}^{2} + \theta_{4}^{2}}$. The $J$-positivity requires $\omega_{34} > 0$ and $\theta_{1}>0$. Then the generic lcK structure is
$$ \left\{\begin{array}{l}
\theta = \theta_{1}  e^{1} + \theta_{3}  e^{3} + \theta_{4}  e^{4} \\
\Omega = \left( \frac{\omega_{34} {\left(\theta_{1} + 1\right)} \theta_{1}}{\theta_{3}^{2} + \theta_{4}^{2}} \right)  e^{12} + \left( -\frac{\omega_{34} \theta_{1} \theta_{4}}{\theta_{3}^{2} + \theta_{4}^{2}} \right)  e^{13} \\
+ \left( \frac{\omega_{34} \theta_{1} \theta_{3}}{\theta_{3}^{2} + \theta_{4}^{2}} \right)  e^{14} + \left( -\frac{\omega_{34} \theta_{1} \theta_{3}}{\theta_{3}^{2} + \theta_{4}^{2}} \right)  e^{23} + \left( -\frac{\omega_{34} \theta_{1} \theta_{4}}{\theta_{3}^{2} + \theta_{4}^{2}} \right)  e^{24} + \omega_{34}  e^{34} \\
\text{with }	\theta_1>0, \theta_3\neq0, \omega_{34} > 0
\end{array}\right. . $$
Applying the complex automorphism
\[\left(\begin{array}{cccc}
1 & 0 & 0 & 0 \\
0 & 1 & 0 & 0 \\
0 & 0 & \frac{\omega_{34} \theta_{1} \theta_{3}}{\theta_{3}^{2} + \theta_{4}^{2}} & \frac{\omega_{34} \theta_{1} \theta_{4}}{\theta_{3}^{2} + \theta_{4}^{2}} \\
0 & 0 & -\frac{\omega_{34} \theta_{1} \theta_{4}}{\theta_{3}^{2} + \theta_{4}^{2}} & \frac{\omega_{34} \theta_{1} \theta_{3}}{\theta_{3}^{2} + \theta_{4}^{2}}
\end{array}\right)
\in \mathrm{Aut}(\mathfrak{rr}_{3,0}^\vee)
\]
we reduce the lck pair $(\Omega, \theta)$ to 
$$ \left\{\begin{array}{l}
\theta= \theta_{1}  e^1 + \omega_{34} \theta_{1}  e^{3} \\
\Omega=\left( \omega_{34} \theta_{1} \frac{ \theta_{1} + 1}{\theta_{3}^{2} + \theta_{4}^{2}} \right)  e^{12}  + \left( \frac{\omega_{34}^{2} \theta_{1}^{2}}{\theta_{3}^{2} + \theta_{4}^{2}} \right)  e^{14} + \left( -\frac{\omega_{34}^{2} \theta_{1}^{2}}{\theta_{3}^{2} + \theta_{4}^{2}} \right)  e^{23} + \left( \frac{\omega_{34}^{3} \theta_{1}^{2}}{\theta_{3}^{2} + \theta_{4}^{2}} \right)  e^{34} \\
\text{with } \theta_{1}>0, \theta_{3}\neq0, \omega_{34}>0
\end{array}\right. . $$
Now we apply the complex automorphism
\[
\left(\begin{array}{cccc}
1 & 0 & 0 & 0 \\
0 & 1 & 0 & 0 \\
0 & 0 & \frac{\theta_{3}^{2} + \theta_{4}^{2}}{\omega_{34}^{2} \theta_{1}^{2}} & 0 \\
0 & 0 & 0 & \frac{\theta_{3}^{2} + \theta_{4}^{2}}{\omega_{34}^{2} \theta_{1}^{2}}
\end{array}\right)
\in \mathrm{Aut}(\mathfrak{rr}_{3,0}^\vee)
\]
and we obtain the lck structure
$$ \left\{\begin{array}{l}
\theta= \delta  e^{1} +  \frac{\sigma}{\delta}  e^{3} \\
\Omega=\frac{ \delta(\delta +1) }{\sigma}   e^{12} +  e^{14} -  e^{23} +  \frac{\sigma}{\delta^2}  e^{34} \\
\text{with } \delta>0, \sigma>0
\end{array}\right., $$
where we denoted $\sigma=\frac{\theta_{3}^2+\theta_{4}^2}{\omega_{34}}$ and $\delta=\theta_{1}$.
This lck structure cannot be further reduced. Indeed, the generic automorphism transforms the coefficient of $\theta$ along $e^4$ to $-a_{34}\frac{\sigma}{\delta}$, whence we chose $a_{34} = 0$. The coefficient of $\Omega$ along $e^{14}$ is transformed to $a_{33}$: we then choose $a_{33}=1$, getting the identity.

\begin{rmk}
We determine now which of these lcK structures on $\mathfrak{rr}_{3,0}$ are of Vaisman type. Let $A=a_1e_1+a_2e_2+a_3e_3+a_4e_4$. We determine $a_i$ such that $\theta(A)=1$ and $A \in (\ker \theta)^\perp$, that is, $\Omega(A,Jx)=0$ for any $x\in \ker \theta$. In this case we obtain that 
$A=\frac{\delta}{\sigma}e_3\in \mathcal{Z}(\mathfrak g)$. Therefore, it follows from Lemma \ref{Vaisman} that all the lcK structures above are of Vaisman type.
\end{rmk}

\subsection{$\mathfrak{rr}_{3,1}$}
Consider the Lie algebra $\mathfrak{rr}_{3,1} = (0,-12,-13,0)$. Consider the complex structure associated, in the chosen coframe, to the matrix
$$
J =
\left(\begin{array}{cccc}
0 & 0 & 0 & -1 \\
0 & 0 & 1 & 0 \\
0 & -1 & 0 & 0 \\
1 & 0 & 0 & 0
\end{array}\right)
\in \mathrm{Aut}(\mathfrak{rr}_{3,1}^\vee).
$$

The generic lcs structures are the following: either
$$
\left\{\begin{array}{l}
\theta = -e^1 \\
\Omega = \omega_{12} e^1 \wedge e^2 + \omega_{13} e^1 \wedge e^3 + \omega_{14} e^1 \wedge e^4 + \omega_{24} e^2 \wedge e^4 + \omega_{34} e^3 \wedge e^4 \\
\text{with } \omega_{12}\omega_{34}-\omega_{13}\omega_{24}\neq0
\end{array}\right. $$
or
$$
\left\{\begin{array}{l}
\theta = -2e^1 \\
\Omega = \omega_{12} e^1 \wedge e^2 + \omega_{13} e^1 \wedge e^3 + \omega_{14} e^1 \wedge e^4 + \omega_{23} e^2 \wedge e^3 \\
\text{with } \omega_{14}\omega_{23}\neq0
\end{array}\right. .
$$

There is no lcK structure with Lee form $\theta=-e^1$. Indeed, the corresponding lcs structures are never $J$-positive, since $\Omega(e^2,Je^2)=0$.

We consider lcK structures with Lee form $\theta=-2e^1$. The $J$-invariance of $\Omega$ requires $\omega_{12}=0$ and $\omega_{13}=0$. Therefore we are reduced to $\Omega=\omega_{14}e^{14} + \omega_{23}e^{23}$. The $J$-positivity requires $\omega_{14}>0$ and $\omega_{23}<0$. Finally, the generic lcK structure is
$$
\left\{\begin{array}{l}
\theta = -2e^1 \\
\Omega = \omega_{14} e^1 \wedge e^4 + \omega_{23} e^2 \wedge e^3 \\
\text{with } \omega_{14}>0, \omega_{23}<0
\end{array}\right. .
$$

The generic automorphisms of $(\mathfrak{rr}_{3,1}, J)$ are associated to
$$
\left(\begin{array}{cccc}
1 & 0 & 0 & 0 \\
0 & a_{22} & a_{23} & 0 \\
0 & -a_{23} & a_{22} & 0 \\
0 & 0 & 0 & 1
\end{array}\right) \in \mathrm{Aut}(\mathfrak{rr}_{3,1}^\vee), \qquad \text{with } a_{22}^{2} + a_{23}^{2}\neq0.
$$
For $a_{22}=\frac{1}{\sqrt{-\omega_{23}}}$ and $a_{23}=0$, we apply the automorphism
$$
\left(\begin{array}{cccc}
1 & 0 & 0 & 0 \\
0 & \frac{1}{\sqrt{-\omega_{23}}} & 0 & 0 \\
0 & 0 & \frac{1}{\sqrt{-\omega_{23}}} & 0 \\
0 & 0 & 0 & 1
\end{array}\right) \in \mathrm{Aut}(\mathfrak{rr}_{3,1}^\vee)
$$
to get the normal form
$$
\left\{\begin{array}{l}
\theta = -2e^1 \\
\Omega = \sigma e^1\wedge e^4 - e^2\wedge e^3 \\
\text{with } \sigma > 0
\end{array}\right. .$$

There is no further reduction, since the generic linear complex automorphism transforms $\sigma e^{14} - e^{23}$ into $\sigma  e^{14} - \left( a_{22}^{2} + a_{23}^{2} \right)  e^{23}$, fixing the coefficient along $e^{14}$.

\begin{rmk}\label{rmk:no-vaisman-on-algebra}
The Lie algebras $\mathfrak{rr}_{3,1}$, $\mathfrak{rr}^{\prime}_{3,\gamma}$ with $\gamma\geq0$, $\mathfrak{r}^\prime_2$, $\mathfrak{r}_{4,1}$, $\mathfrak{r}_{4,\alpha,\beta}$ with $\alpha$ and $\beta$ as in Table \ref{table:complex}, $\mathfrak{r}^\prime_{4,\gamma,\delta}$ with $\gamma\in\mathbb R$ and $\delta>0$, do not admit any Vaisman structure, thanks to \cite[Structure Theorem]{ornea-verbitsky-MRL}, \cite[Corollary 3.5]{ornea-verbitsky-JGP}, and to the classification of $3$-dimensional Sasaki Lie algebras, see \cite{cho-chun, geiges}.
\end{rmk}

\subsection{$\mathfrak{rr}'_{3,0}$}
Consider the Lie algebra $\mathfrak{rr}'_{3,0}=(0,-13,12,0)$, endowed with the complex structure associated to the matrix
$$
J=
\left(\begin{array}{cccc}
0 & 0 & 0 & -1 \\
0 & 0 & -1 & 0 \\
0 & 1 & 0 & 0 \\
1 & 0 & 0 & 0
\end{array}\right)
\in \mathrm{Aut}((\mathfrak{rr}_{3,0}^{\prime})^\vee) .
$$
The generic lcs structure is
$$
\left\{\begin{array}{l}
\theta = \theta_{1}  e^{1} + \theta_{4}  e^{4} \\
\Omega = \omega_{12}  e^{12} + \omega_{13}  e^{13} + \omega_{14}  e^{14} + \left( -\frac{{\left(\omega_{12} \theta_{1} + \omega_{13}\right)} \theta_{4}}{\theta_{1}^{2} + 1} \right)  e^{24} + \left( -\frac{{\left(\omega_{13} \theta_{1} - \omega_{12}\right)} \theta_{4}}{\theta_{1}^{2} + 1} \right)  e^{34} \\
\text{with } {\left(\omega_{12}^{2} + \omega_{13}^{2}\right)} \theta_{4} \neq 0
\end{array}\right. .$$
It is clear that $\Omega$ is never $J$-positive: indeed $\omega(e_2, Je_2)=0$. Then, there is no lcK structure in this case.

\subsection{$\mathfrak{rr}'_{3,\gamma}$ with $\gamma>0$}
Consider the Lie algebra $\mathfrak{rr}'_{3,\gamma}=(0,-\gamma 12-13,12-\gamma 13,0)$. It admits two non-equivalent complex structures.

We first consider the complex structure $J_1$ associated to the matrix
$$
J_1=
\left(\begin{array}{cccc}
0 & 0 & 0 & -1 \\
0 & 0 & -1 & 0 \\
0 & 1 & 0 & 0 \\
1 & 0 & 0 & 0
\end{array}\right)
\in \mathrm{Aut}((\mathfrak{rr}_{3,\gamma}^{\prime})^\vee)
$$
According to \cite[Appendix 6.4 of the arXiv version, pages 36--37]{angella-bazzoni-parton}, the generic lcs structures are the following: either
$$
\left\{\begin{array}{l}
\theta = -2\gamma e^1 \\
\Omega = \omega_{12} e^1 \wedge e^2 + \omega_{13} e^1 \wedge e^3 + \omega_{14} e^1 \wedge e^4 + \omega_{23} e^2 \wedge e^{3} \\
\text{with } \omega_{14}\omega_{23}\neq0, \gamma \neq 0
\end{array}\right. $$
or
$$
\left\{\begin{array}{l}
\theta = \theta_{1}  e^{1} + \theta^{4}  e_{4} \\
\Omega = \omega_{12}  e^{12} + \omega_{13}  e^{13} + \omega_{14}  e^{14} + \left( -\frac{{\left(\omega_{12} (\gamma + \theta_{1}) + \omega_{13}\right)} \theta_{4}}{(\gamma+\theta_{1})^{2} + 1} \right)  e^{24} + \left( -\frac{{\left(\omega_{13} (\gamma+\theta_{1}) - \omega_{12}\right)} \theta_{4}}{(\gamma+\theta_{1})^{2} + 1} \right)  e^{34} \\
\text{with } {\left(\omega_{12}^{2} + \omega_{13}^{2}\right)} \theta_{4} \neq 0
\end{array}\right. .$$

There is no lcK structure with Lee form $\theta=\theta_{1}  e^{1} + \theta^{4}  e_{4}$. Indeed, the corresponding lcs structures are never $J_1$-positive, since $\Omega(e^2,Je^2)=\omega_{23}=0$.

We consider lcK structures with Lee form $\theta=-2\gamma e^1$. The $J_1$-invariance of $\Omega$ requires $\omega_{12}=0$ and $\omega_{13}=0$. Therefore we are reduced to $\Omega=\omega_{14}e^{14} + \omega_{23}e^{23}$. The $J_1$-positivity requires $\omega_{14}>0$ and $\omega_{23}>0$. Finally, the generic lcK structure is
$$
\left\{\begin{array}{l}
\theta = -2\gamma e^1 \\
\Omega = \omega_{14} e^1 \wedge e^4 + \omega_{23} e^2 \wedge e^3 \\
\text{with } \omega_{14}>0, \omega_{23}>0, \gamma \neq 0 
\end{array}\right. .
$$

The generic automorphisms of $(\mathfrak{rr}_{3,\lambda}, J_1)$ are associated to
$$
\left(\begin{array}{cccc}
1 & 0 & 0 & 0 \\
0 & a_{22} & a_{23} & 0 \\
0 & -a_{23} & a_{22} & 0 \\
0 & 0 & 0 & 1
\end{array}\right) \in \mathrm{Aut}((\mathfrak{rr}_{3,\gamma}^{\prime})^\vee) , \qquad \text{with } a_{22}^{2} + a_{23}^{2}\neq0.
$$
For $a_{22}=\frac{1}{\sqrt{\omega_{23}}}$ and $a_{23}=0$, we apply the automorphism
$$
\left(\begin{array}{cccc}
1 & 0 & 0 & 0 \\
0 & \frac{1}{\sqrt{\omega_{23}}} & 0 & 0 \\
0 & 0 & \frac{1}{\sqrt{\omega_{23}}} & 0 \\
0 & 0 & 0 & 1
\end{array}\right)
\in \mathrm{Aut}((\mathfrak{rr}_{3,\gamma}^{\prime})^\vee)
$$
to get the lcK form
$$
\left\{\begin{array}{l}
\theta = -2\gamma e^1 \\
\Omega = \sigma e^1\wedge e^4 + e^2\wedge e^3 \\
\text{with } \sigma > 0, \gamma \neq 0
\end{array}\right. .$$

There is no further reduction, since the generic linear complex automorphism transforms $\sigma e^{14} + e^{23}$ into $\sigma  e^{14} + \left( a_{22}^{2} + a_{23}^{2} \right)  e^{23}$, fixing the coefficient along $e^{14}$.

Next we consider the complex structure $J_2$ associated to the matrix 
$$
J_2=
\left(\begin{array}{cccc}
0 & 0 & 0 & -1 \\
0 & 0 & 1 & 0 \\
0 & -1 & 0 & 0 \\
1 & 0 & 0 & 0
\end{array}\right)
\in \mathrm{Aut}((\mathfrak{rr}_{3,\gamma}^{\prime})^\vee),
$$
the only difference in this case is that $\Omega(e_2,J_2e_2)=-\omega_{23}>0$, then $\omega_{23}<0$. In the same way as above we get the final lcK form
$$
\left\{\begin{array}{l}
\theta = -2\gamma e^1 \\
\Omega = \sigma e^1\wedge e^4 - e^2\wedge e^3 \\
\text{with } \sigma > 0, \gamma \neq 0
\end{array}\right. .$$

\begin{rmk}
This algebra does not admit any Vaisman structure, see Remark \ref{rmk:no-vaisman-on-algebra}.
\end{rmk}

\subsection{$\mathfrak{r}_2\mathfrak{r}_2$}
Consider the Lie algebra $\mathfrak{r}_2\mathfrak{r}_2 = (0,-12, 0, -34)$ with the complex structure defined as
$$
J = \left(\begin{array}{cccc}
0 & -1 & 0 & 0 \\
1 & 0 & 0 & 0 \\
0 & 0 & 0 & -1 \\
0 & 0 & 1 & 0
\end{array}\right)
\in \mathrm{Aut}((\mathfrak{r}_{2}\mathfrak{r}_2)^\vee)
$$
in terms of the chosen coframe.  A generic automorphism for $(\mathfrak{r}_2\mathfrak{r}_2^\vee,J)$ is associated to the identity matrix or to 
\begin{equation}\label{aut_r2r2}
\left(\begin{array}{cccc}
	0 & 0 & 1 & 0 \\
	0 & 0 & 0 & 1 \\
	1 & 0 & 0 & 0 \\
	0 & 1 & 0 & 0
\end{array}\right) \in \mathrm{Aut}((\mathfrak{r}_{2}\mathfrak{r}_2)^\vee) .
\end{equation}

As in \cite[Appendix 6.5 of the arXiv version, pages 38--39]{angella-bazzoni-parton}, the generic lcs structures are either
$$ \left\{\begin{array}{l}
\theta = \theta_3 e^3 \\
\Omega = -\frac{\omega_{23}}{\theta_3} e^{12} + \omega_{13} e^{13} + \omega_{14} e^{14} + \omega_{23} e^{23} + \omega_{34} e^{34} \\
\text{	with } \omega_{14}(\theta_3 + 1)=0,  \omega_{23}(\omega_{14}\theta_3-\omega_{34})\neq0, \theta_3 \neq 0
\end{array}\right. , $$
or
$$ \left\{\begin{array}{l}
\theta = \theta_1 e^1 \\
 \Omega = \omega_{12} e^{12} + \omega_{13} e^{13} + \omega_{14} e^{14} + \omega_{23} e^{23} + \frac{\omega_{14}}{\theta_1} e^{24} \\
\text{	with }  \omega_{23}(\theta_1 + 1)=0, \omega_{14}(\omega_{12} + \theta_{1}\omega_{23})\neq0, \theta_1 \neq 0
\end{array}\right. , $$
or
$$ \left\{\begin{array}{l}
\theta = -e^1 -e^3\\
\Omega =  \omega_{13} e^{13} + \omega_{14} e^{14} + \omega_{23} e^{23} + \omega_{24} e^{24}  \\
\text{	with } \omega_{13}\omega_{14}- \omega_{14}\omega_{23}\neq0
\end{array}\right. , $$
or 
$$ \left\{\begin{array}{l}
\theta = \theta_{1}e^1 +\theta_{3}e^3\\
\Omega = -\frac{1+\theta_1}{\theta_{3}}\omega_{23}e^{12}+ \omega_{13} e^{13} + \omega_{14} e^{14} + \omega_{23} e^{23} + \frac{\theta_{3}+1}{\theta_{1}} \omega_{14} e^{34}  \\
\text{	with }\theta_{1}\theta_{3}\neq0, \theta_{1}+\theta_{3}\neq -1, \omega_{14}\omega_{23}\neq0
\end{array}\right. . $$

First we consider the case $\theta=\theta_{3}e^3$. The lcK condition yields the generic form
$$ \left\{\begin{array}{l}
\theta = - e^3 \\
\Omega = \omega_{12}(  e^{12} - e^{14} + e^{23} ) +  \omega_{34}  e^{34} \\
\text{	with } \omega_{34}>\omega_{12}>0
\end{array}\right. . $$
	
The only complex automorphism fixing the Lee form is the identity. Whence the above form is the generic lcK form up to linear equivalence.

\begin{rmk}
We determine now which of these lcK structures on $\mathfrak{r}_{2}\mathfrak{r}_2$ with Lee form $\theta=-e^3$ are of Vaisman type. Let $A=a_1e_1+a_2e_2+a_3e_3+a_4e_4$. We determine $a_i$ such that $\theta(A)=1$ and $A \in (\ker \theta)^\perp$. In this case we obtain that 
$A=-e_1+\frac{\omega_{34}}{\omega_{12}}e_2-e_3+\frac{\omega_{34}}{\omega_{12}}$. Therefore
$$ \mathrm{ad}_A = \left(\begin{matrix}
0 & 0 & 0 & 0 \\
-\frac{\omega_{34}}{\omega_{12}} & -1 & 0 & 0 \\
0 & 0 & 0 & 0 \\
0 & 0 & - \frac{\omega_{34}}{\omega_{12}} & -1
\end{matrix}\right).$$
From Lemma \ref{Vaisman}, it follows that none of the lcK structures above are of Vaisman type.
\end{rmk}

We now assume $\theta=\theta_{1}e^1$. Requiring $\Omega$ to be $J$-positive and $J$-invariant, we obtain 
$$ \left\{\begin{array}{l}
\theta = - e^1 \\
\Omega = \omega_{12}e^{12} + \omega_{34}(- e^{14} + e^{23} +  e^{34})  \\
\text{	with }\omega_{12}>\omega_{34}>0
\end{array}\right. . $$
In the same way as above there is no further reduction.

\begin{rmk}
We determine now which of these lcK structures on $\mathfrak{r}_{2}\mathfrak{r}_2$ with Lee form $\theta=-e^1$ are of Vaisman type. Let $A=a_1e_1+a_2e_2+a_3e_3+a_4e_4$. We determine $a_i$ such that $\theta(A)=1$ and $A \in (\ker \theta)^\perp$. In this case we obtain that 
$A=-e_1-e_3$. Therefore
$$ \mathrm{ad}_A = \left(\begin{matrix}
0 & 0 & 0 & 0 \\
0 & -1 & 0 & 0 \\
0 & 0 & 0 & 0 \\
0 & 0 & 0 & -1
\end{matrix}\right).$$
From Lemma \ref{Vaisman}, it follows that none of the lcK structures above are of Vaisman type.
\end{rmk}

If $\theta=-e^1-e^3$ does not produce lcK forms since $\Omega(e_1,Je_1)=0$, therefore $\Omega$ is never $J$-positive.

Finally if $\theta = \theta_{1}e^1 +\theta_{3}e^3$, requiring the lcK conditions we obtain
$$ \left\{\begin{array}{l}
\theta = \sigma e^1 +\tau e^3\\
\Omega = \mu(\frac{1+\sigma}{\tau}e^{12} + e^{14} -e^{23} + \frac{\tau+1}{\sigma} e^{34})  \\
\text{	with }\sigma\tau\neq0, \sigma+\tau\neq -1, \mu\neq0, \frac{\mu(1+\sigma)}{\tau}>0, \frac{\mu(\tau+1)}{\sigma}>0 , \frac{\sigma + \tau + 1}{\sigma \tau} > 0
\end{array}\right. , $$
where $\sigma=\theta_{1}, \tau=\theta_{3}, \mu=\omega_{14}$. Applying the automorphism \eqref{aut_r2r2} we can assume that $\sigma\leq\tau$.

\begin{rmk}
We determine now which of these lcK structures on $\mathfrak{r}_{2}\mathfrak{r}_2$ with Lee form $\theta=\sigma e^1 +\tau e^3$ are of Vaisman type. Let $A=a_1e_1+a_2e_2+a_3e_3+a_4e_4$. We determine $a_i$ such that $\theta(A)=1$ and $A \in (\ker \theta)^\perp$. In this case we obtain that 
$A=\frac{1+2\tau}{\sigma+2\sigma\tau+\tau}e_1 + \frac{1}{\sigma+2\sigma\tau+\tau}e_3$. Therefore
$$ \mathrm{ad}_A = \left(\begin{matrix}
0 & 0 & 0 & 0 \\
0 & \frac{1+2\tau}{\sigma+2\sigma\tau+\tau} & 0 & 0 \\
0 & 0 & 0 & 0 \\
0 & 0 & 0 & \frac{1}{\sigma+2\sigma\tau+\tau}
\end{matrix}\right).$$
From Lemma \ref{Vaisman}, it follows that none of the lcK structures above are of Vaisman type.
\end{rmk}

\subsection{$\mathfrak{r}'_2$}

Consider the Lie algebra $\mathfrak{r}'_2 = (0,0, -13+24, -14-23)$. 

As in \cite[Appendix 6.6 of the arXiv version, pages 43--45]{angella-bazzoni-parton}, the generic lcs structure is either
$$ \left\{\begin{array}{l}
\theta = \theta_1 e^1 + \theta_2 e^2 \\
\Omega = \omega_{12}  e^{12} + \omega_{13}  e^{13} + \left( \frac{\omega_{24} {\left(\theta_{1} + 1\right)} + \omega_{13}}{\theta_{2}} \right)  e^{14} + \left( \frac{{\left(\theta_{2}^{2} + 1\right)} \omega_{13} + \omega_{24} {\left(\theta_{1} + 1\right)}}{{\left(\theta_{1} + 1\right)} \theta_{2}} \right)  e^{23} + \omega_{24}  e^{24} \\
\text{	with } \omega_{13}^2+\omega_{14}^2\neq0 , \theta_1\neq-1, \theta_2\neq0
\end{array}\right. , $$
or
$$ \left\{\begin{array}{l}
\theta = -2 e^1 \\
\Omega = \omega_{12}  e^{12} + \omega_{13}  (e^{13}+e^{24}) + \omega_{14}(e^{14}-e^{23})+\omega_{34}e^{34} \\
\text{	with } \omega_{12} \omega_{34} - \omega_{13}^2-\omega_{14}^2 \neq0, \omega_{34}\neq0
\end{array}\right. , $$
or
$$ \left\{\begin{array}{l}
\theta = \theta_{1} e^1 \\
\Omega = \omega_{12}  e^{12} -(\theta_{1}+1) \omega_{24} e^{13}+(\theta_{1}+1) \omega_{23}e^{14} +\omega_{23}e^{23}+\omega_{24}e^{24} \\
\text{	with } \omega_{23}^2+\omega_{24}^2\neq0, \theta_1 \neq 0
\end{array}\right. . $$

According to \cite{ovando} this Lie algebra admits several different complex structures given by
$$
J_1 = \left(\begin{array}{cccc}
0 & 0 & -1 & 0 \\
0 & 0 & 0 & -1 \\
1 & 0 & 0 & 0 \\
0 & 1 & 0 & 0
\end{array}\right)
\in \mathrm{Aut}((\mathfrak{r}_{2}^\prime)^\vee)
, \quad 
J_2 = \left(\begin{array}{cccc}
-a & -b & 0 & 0 \\
\frac{a^{2} + 1}{b} & a & 0 & 0 \\
0 & 0 & 0 & -1 \\
0 & 0 & 1 & 0
\end{array}\right)
\in \mathrm{Aut}((\mathfrak{r}_{2}^\prime)^\vee)
$$
in terms of the chosen coframe.

We study first the complex structure $J_1$. The only complex automorphisms of $(\mathfrak{r}'_2, J_1)$ are
\begin{equation}\label{aut_r2'}
\left( \begin{array}{cccc}
1 & 0 & 0 & 0 \\
0 & \pm1 & 0 & 0 \\
0 & 0 & 1 & 0 \\
0 & 0 & 0 & \pm1
\end{array}\right)
\in \mathrm{Aut}((\mathfrak{r}_{2}^\prime)^\vee).
\end{equation}

If the lcs structure is as in the first case, then the generic lcK form is
$$ \left\{\begin{array}{l}
\theta = \theta_{1}e^1+\theta_{2}e^2 \\
\Omega = \omega_{13}\left( e^{13} + \frac{\theta_{2}}{\theta_{1}}( e^{14} + e^{23}) +  \frac{\theta_{2}^2-\theta_{1}}{\theta_{1}(\theta_{1}+1)} e^{24}\right) \\
\text{	with } \omega_{13}>0,  \theta_2\neq0, \theta_1+1<0
\end{array}\right. . $$
Applying the automorphism \eqref{aut_r2'} we can assume $\theta_{2}>0$.

We now consider the second case for the lcs form. Requiring $\Omega$ to be $J_1$-positive and $J_1$-invariant we obtain the lcK form
$$ 
\left\{\begin{array}{l}
\theta = -2e^1 \\
\Omega = \omega_{12} ( e^{12} +e^{34}) +  \omega_{13} (e^{13} +  e^{24}) \\
\text{	with } \omega_{13}>0, \omega_{13}^2-\omega_{12}^2>0
\end{array}\right. . $$
Applying the automorphism \eqref{aut_r2'} we can assume $\omega_{12}>0$.

In the last case, the generic lcK form is
$$ \left\{\begin{array}{l}
\theta = \theta_{1}e^1\\
\Omega = \omega_{24}\left( -(\theta_{1}+1)e^{13} + e^{24}\right) \\
\text{	with } \omega_{24}>0, \theta_{1}+1<0
\end{array}\right. . $$
There is no further reduction, since $\Omega$ is fixed by a generic automorphism \eqref{aut_r2'}.

Now we consider the second complex structure $J_2$.

The complex automorphisms of $(\mathfrak{r}'_2, J_2)$ are
\begin{equation}\label{aut_r2'J2}
\left(\begin{array}{cccc}
1 & 0 & 0 & 0 \\
0 & 1 & 0 & 0 \\
a_{13} & -a_{14} & a_{33} & -a_{34} \\
a_{14} & a_{13} & a_{34} & a_{33}
\end{array}\right)
\in \mathrm{Aut}((\mathfrak{r}_{2}^\prime)^\vee),
\end{equation}
with $a_{33}^{2} + a_{34}^{2}\neq0$, and moreover, if $(a,b)\neq(0,1)$, then $a_{13}=0, a_{14}=0$.

By $J_2$-positivity, we are reduced to only one possibility for the lcs structure, namely,
$$ \left\{\begin{array}{l}
\theta = -2 e^1 \\
\Omega = \omega_{12}  e^{12} + \omega_{13}  (e^{13}+e^{24}) + \omega_{14}(e^{14}-e^{23})+\omega_{34}e^{34} \\
\text{	with } \omega_{12} \omega_{34} - \omega_{13}^2-\omega_{14}^2 \neq0, \omega_{34}\neq0
\end{array}\right. . $$

In the general case $(a,b)\neq (0,1)$, the conditions for $J_2$-invariance and $J_2$-positivity yield to the general lcK form
$$ \left\{\begin{array}{l}
\theta = -2 e^1 \\
\Omega = \omega_{12}  e^{12} +\omega_{34}e^{34} \\
\text{	with } b\omega_{12}>0, \omega_{34}>0
\end{array}\right. . $$

Using the automorphism with $a_{33}=\frac{1}{\sqrt{\omega_{34}}}$ and $a_{34}=0$, we get the normal form
$$ \left\{\begin{array}{l}
\theta = -2 e^1 \\
\Omega = \omega_{12}  e^{12} + e^{34} \\
\text{	with } b\omega_{12}>0
\end{array}\right. , $$
and no further reduction is possible.

In the particular case $(a,b)=(0,1)$, the generic lcK structure is
$$ \left\{\begin{array}{l}
\theta = -2 e^1 \\
\Omega = \omega_{12}  e^{12} + \omega_{13}  (e^{13}+e^{24}) + \omega_{14}(e^{14}-e^{23})+\omega_{34}e^{34} \\
\text{	with } \omega_{12} >0, \omega_{34} >0, 
\omega_{12}\omega_{34}-\omega_{13}^2 - \omega_{14}^2  >0
\end{array}\right. . $$

We use the automorphisms with $a_{13}=-\frac{\omega_{14}}{\omega_{34}}$, $a_{14}=\frac{\omega_{13}}{\omega_{34}}$, $a_{33}=\frac{1}{\sqrt{\omega_{34}}}$, and $a_{34}=0$ to get
$$ \left\{\begin{array}{l}
\theta = -2 e^1 \\
\Omega = \sigma  e^{12} + \omega_{34}e^{34} \\
\text{	with } \sigma >0
\end{array}\right. , $$
and no further reduction is possible.

\begin{rmk}
This algebra does not admit any Vaisman structure, see Remark \ref{rmk:no-vaisman-on-algebra}.
\end{rmk}

\subsection{$\mathfrak{r}_{4,1}$}
Consider the Lie algebra $\mathfrak{r}_{4,1} = (14,24+34,34,0)$ with the complex structure defined as
$$
J = \left(\begin{array}{cccc}
0 & -1 & 0 & 0 \\
1 & 0 & 0 & 0 \\
0 & 0 & 0 & 1 \\
0 & 0 & -1 & 0
\end{array}\right)
\in \mathrm{End}(\mathfrak{r}_{4,1}^\vee)
$$
in terms of the chosen coframe.

According to \cite[Appendix 6.9 of the arXiv version, page 51]{angella-bazzoni-parton} the generic lcs structure is
$$
\left\{\begin{array}{l}
\theta = -2e_{4} \\
\Omega = \omega_{13}  e^{13} + \omega_{14} e^{14} +  \omega_{23}e^{23} + \omega_{24} e^{24}+ \omega_{34}e^{34} \\
\text{with }\omega_{13}\omega_{24}-\omega_{14}\omega_{23} \neq 0
\end{array}\right. .$$

It is clear that $\Omega$ is never $J$-positive: indeed $\Omega(e_1, Je_1)=0$. Then, there is no lcK structure for this Lie algebra.

\subsection{$\mathfrak{r}_{4,\alpha,1}$ with $\alpha\neq 0,1$}

Consider the Lie algebra $\mathfrak{r}_{4,\alpha,1} = (14,\alpha24,34,0)$ with the complex structure defined as
$$
J = \left(\begin{array}{cccc}
0 & 0 & -1 & 0 \\
0 & 0 & 0 & 1 \\
1 & 0 & 0 & 0 \\
0 & -1 & 0 & 0
\end{array}\right) \in \mathrm{End}(\mathfrak{r}_{4,\alpha,1}^\vee)
$$
in terms of the chosen coframe.

According to \cite[Appendix 6.10 of the arXiv version, page 54]{angella-bazzoni-parton} the generic lcs structures are either 
$$
\left\{\begin{array}{l}
\theta = -(1+\alpha)e_{4} \\
\Omega = \omega_{12}  e^{12} + \omega_{14} e^{14} +  \omega_{23}e^{23} + \omega_{24} e^{24}+ \omega_{34}e^{34} \\
\text{with }\omega_{12}\omega_{34}+\omega_{14}\omega_{23} \neq 0
\end{array}\right. ,$$
only when $\alpha\neq -1$,
or
$$
\left\{\begin{array}{l}
\theta = -2e_{4} \\
\Omega = \omega_{13}  e^{13} + \omega_{14} e^{14} + \omega_{24} e^{24}+ \omega_{34}e^{34} \\
\text{with }\omega_{13}\omega_{24}\neq 0
\end{array}\right. .$$

In the first case $\theta = -(1+\alpha)e_{4}$, we have that $\Omega$ is never $J$-positive: indeed $\Omega(e_1, Je_1)=0$. Then, there is no lcK structure in this case. In the second case $\theta = -2e_{4}$, the J-invariance of $\Omega$ requires $\omega_{14}=0$ and $\omega_{34}=0$ and $J$-positive implies $\omega_{13}>0$ and $\omega_{24}<0$. Finally the generic lcK structure is 
$$
\left\{\begin{array}{l}
\theta = -2e_{4} \\
\Omega = \omega_{13}  e^{13} + \omega_{24} e^{24} \\
\text{with }\omega_{13}>0, \omega_{24}<0
\end{array}\right. .$$

The generic automorphisms of ($\mathfrak{r}_{4,\alpha,1},J)$ with $\alpha\neq 0,1$ are associated to
$$\left(\begin{array}{cccc}
	a_{11} & 0 & a_{13} & 0 \\
	0 & 1 & 0 & 0 \\
	-a_{13} & 0 & a_{11} & 0 \\
	0 & 0 & 0 & 1
\end{array}\right)
\in \mathrm{End}(\mathfrak{r}_{4,\alpha,1}^\vee)
\text{ with }a_{11}^{2} + a_{13}^{2}\neq0.
$$
For $a_{11}=\frac{1}{\sqrt{\omega_{13}}}$ and $a_{13}=0$, we apply the automorphism
$$
\left(\begin{array}{cccc}
\frac{1}{\sqrt{\omega_{13}}} & 0 & 0 & 0 \\
0 & 1 & 0 & 0 \\
0 & 0 & \frac{1}{\sqrt{\omega_{13}}} & 0 \\
0 & 0 & 0 & 1
\end{array}\right)
\in \mathrm{End}(\mathfrak{r}_{4,\alpha,1}^\vee)
$$
to get the lcK form
$$
\left\{\begin{array}{l}
\theta = -2 e^4 \\
\Omega = e^1\wedge e^3 + \sigma e^2\wedge e^4 \\
\text{with } \sigma < 0
\end{array}\right. .$$

There is no further reduction, since the generic linear complex automorphism transforms $ e^{13} +\sigma e^{24}$ into $  \left( a_{11}^{2} + a_{13}^{2} \right)e^{13} +\sigma  e^{24}$, fixing the coefficient along $e^{24}$.

\begin{rmk}
This algebra does not admit any Vaisman structure, see Remark \ref{rmk:no-vaisman-on-algebra}.
\end{rmk}

\subsection{$\mathfrak{r}_{4,\alpha,\alpha}$ with $\alpha\not\in\{ 0,1\}$}

Consider the Lie algebras $\mathfrak{r}_{4,\alpha,\alpha} = (14,\alpha24,\alpha34,0)$ for $\alpha\not\in\{0,1\}$ (the case $\alpha=-1$ is sometimes also denoted as $\hat{\mathfrak{r}}_{4,-1} = \mathfrak{r}_{4,\alpha=-1,\alpha=-1}=(14,-24,-34,0)$),  with the complex structure defined as
$$
J = \left(\begin{array}{cccc}
0 & 0 & 0 & 1 \\
0 & 0 & -1 & 0 \\
0 & 1 & 0 & 0 \\
-1 & 0 & 0 & 0
\end{array}\right)
\in \mathrm{End}(\mathfrak{r}_{4,\alpha,\alpha}^\vee)
$$
in terms of the chosen coframe.

According to \cite[Appendix 6.10 of the arXiv version, page 52, and Appendix 6.11 of the arXiv version, page 55]{angella-bazzoni-parton} the generic lcs structure for both Lie algebras is 
$$
\left\{\begin{array}{l}
\theta = -2\alpha e_{4} \\
\Omega =  \omega_{14} e^{14} +  \omega_{23}e^{23} + \omega_{24} e^{24}+ \omega_{34}e^{34} \\
\text{with }\omega_{14}\omega_{23} \neq 0
\end{array}\right. .$$
Requiring $J$-positive we obtain $\Omega(e_1,Je_1)=-\omega_{14}>0$, $\Omega(e_2,Je_2)=\omega_{23}>0$. Assuming $\Omega$ is $J$-positive, we obtain $\omega_{34}=\omega_{12}=0$ and $\omega_{24}=-\omega_{13}=0$. Therefore the generic lcK structure for these Lie algebras is
$$
\left\{\begin{array}{l}
\theta = -2\alpha e_{4} \\
\Omega =  \omega_{14} e^{14} +  \omega_{23}e^{23} \\
\text{with }\omega_{14}<0, \omega_{23}>0
\end{array}\right. .$$

The generic complex automorphisms for these Lie algebras  are associated to
$$\left(\begin{array}{cccc}
1 & 0 & 0 & 0 \\
0 & a_{22} & a_{23} & 0 \\
0 & -a_{23} & a_{22} & 0 \\
0 & 0 & 0 & 1
\end{array}\right)
\in \mathrm{End}(\mathfrak{r}_{4,\alpha,\alpha}^\vee)
\text{ with }a_{22}^{2} + a_{23}^{2} \neq 0
$$
For $a_{22}=\frac{1}{\sqrt{\omega_{23}}}$ and $a_{23}=0$, we apply the automorphism
$$
\left(\begin{array}{cccc}
1 & 0 & 0 & 0 \\
0 & \frac{1}{\sqrt{\omega_{23}}} & 0 & 0 \\
0 & 0 & \frac{1}{\sqrt{\omega_{23}}} & 0 \\
0 & 0 & 0 & 1
\end{array}\right)
\in \mathrm{End}(\mathfrak{r}_{4,\alpha,\alpha}^\vee)
$$
to get the lcK form
$$
\left\{\begin{array}{l}
\theta = -2\alpha e^4 \\
\Omega = \sigma e^1\wedge e^4 + e^2\wedge e^3 \\
\text{with } \sigma < 0
\end{array}\right. .$$

There is no further reduction, since the generic linear complex automorphism transforms $ \sigma e^{14} + e^{23}$ into $  \sigma e^{14} + \left( a_{22}^{2} + a_{23}^{2} \right)e^{23}$, fixing the coefficient along $e^{14}$.

\begin{rmk}
This algebra does not admit any Vaisman structure, see Remark \ref{rmk:no-vaisman-on-algebra}.
\end{rmk}

\subsection{$\mathfrak{r'}_{4,\gamma,\delta}$ with $\delta>0$}

Consider the Lie algebra $\mathfrak{r'}_{4,\gamma,\delta} = (14,\gamma24+\delta34,-\delta24+\gamma34,0)$. 
According to \cite[Appendix 6.12 of the arXiv version, page 56]{angella-bazzoni-parton} the generic lcs structure is 
$$
\left\{\begin{array}{l}
\theta = -2\gamma e_{4} \\
\Omega =  \omega_{14}e^{14} + \omega_{23} e^{23}  + \omega_{24} e^{24}+ \omega_{34}e^{34} \\
\text{with }\omega_{14}\omega_{23} \neq 0 
\end{array}\right. , $$
only when $\gamma\neq0$.

According to \cite{ovando} this Lie algebra admits two not equivalent complex structures. We consider first the complex structure defined as
$$
J_1 = \left(\begin{array}{cccc}
0 & 0 & 0 & 1 \\
0 & 0 & -1 & 0 \\
0 & 1 & 0 & 0 \\
-1 & 0 & 0 & 0
\end{array}\right)
\in \mathrm{End}((\mathfrak{r}^\prime_{4,\gamma,\delta})^\vee)
$$
We impose now $\Omega$ to be $J_1$-invariant and $J_1$-positive and we reduce the generic lcs structure to
$$\left\{\begin{array}{l}
	\theta = -2\gamma e_{4} \\
	\Omega =  \omega_{14} e^{14}  + \omega_{23} e^{23} \\
	\text{with }\omega_{14}<0,\omega_{23} > 0
\end{array}\right. ,$$
only when $\gamma\neq0$.
The generic automorphisms of ($\mathfrak{r'}_{4,\gamma,\delta},J_1)$ with $\alpha\neq 0,1$ are associated to
\begin{equation}\label{aut_r'4.gamma.delta}
\left(\begin{array}{cccc}
1 & 0 & 0 & 0 \\
0 & a_{22} & a_{23} & 0 \\
0 & -a_{23} & a_{22} & 0 \\
0 & 0 & 0 & 1
\end{array}\right)
\in \mathrm{End}((\mathfrak{r}^\prime_{4,\gamma,\delta})^\vee)
\text{ with }a_{22}^{2} + a_{23}^{2}\neq0.
\end{equation}

For $a_{22}=\frac{1}{\sqrt{\omega_{23}}}$ and $a_{23}=0$, we apply the automorphism
$$
\left(\begin{array}{cccc}
1 & 0 & 0 & 0 \\
0 & \frac{1}{\sqrt{\omega_{23}}} & 0 & 0 \\
0 & 0 & \frac{1}{\sqrt{\omega_{23}}} & 0 \\
0 & 0 & 0 & 1
\end{array}\right)
\in \mathrm{End}((\mathfrak{r}^\prime_{4,\gamma,\delta})^\vee)
$$
to get the lcK form
$$
\left\{\begin{array}{l}
\theta = - 2\gamma e^4 \\
\Omega = \sigma e^1\wedge e^4 + e^2\wedge e^3\\
\text{with } \sigma < 0
\end{array}\right. ,$$
only when $\gamma\neq0$.

There is no further reduction, since the generic linear complex automorphism transforms $ \sigma e^{14} + e^{23}$ into $\sigma  e^{14} + \left( a_{22}^{2} + a_{23}^{2} \right) e^{23}$, fixing the coefficient along $e^{14}$.

Next we consider the second complex structure given by 
$$
J_2= \left(\begin{array}{cccc}
0 & 0 & 0 & 1 \\
0 & 0 & 1 & 0 \\
0 & -1 & 0 & 0 \\
-1 & 0 & 0 & 0
\end{array}\right)
\in \mathrm{End}((\mathfrak{r}^\prime_{4,\gamma,\delta})^\vee)
$$
in terms of the chosen coframe. Requiring $\Omega$ to be $J_2$-positive we get $\omega_{23}<0$, this is the only difference with the case $J_1$. Also the complex automorphisms are the same. Taking the automorphism
$$
\left(\begin{array}{cccc}
1 & 0 & 0 & 0 \\
0 & \frac{1}{\sqrt{-\omega_{23}}} & 0 & 0 \\
0 & 0 & \frac{1}{\sqrt{-\omega_{23}}} & 0 \\
0 & 0 & 0 & 1
\end{array}\right)
\in \mathrm{End}((\mathfrak{r}^\prime_{4,\gamma,\delta})^\vee)
$$
the generic lcK form for this Lie algebra reduces to
$$
\left\{\begin{array}{l}
\theta = -2\gamma e^4 \\
\Omega = \sigma e^1\wedge e^4 - e^2\wedge e^3\\
\text{with } \sigma < 0
\end{array}\right. ,$$
only when $\gamma\neq0$.

There is no further reduction: the generic automorphism trasforms $\sigma e^{14}-e^{23}$ into $\sigma e^{14}-(a_{22}^2+a_{23}^2)e^{23}$.

\begin{rmk}
This algebra does not admit any Vaisman structure, see Remark \ref{rmk:no-vaisman-on-algebra}.
\end{rmk}

\subsection{$\mathfrak{d}_4$}

Consider the Lie algebra $\mathfrak{d}_4 = (14,-24,-12,0)$. According to \cite[Appendix 6.13 of the arXiv version, pages 56--57]{angella-bazzoni-parton} the generic lcs structure are either
$$
\left\{\begin{array}{l}
\theta = \theta_4 e_{4} \\
\Omega =  \omega_{12}(e^{12}-\theta_4e^{34}) + \omega_{14} e^{14}  + \omega_{24} e^{24}\\
\text{with } \theta_4 \not\in \{-1, 0, 1\} , \omega_{12}\neq0 
\end{array}\right. , $$
or
$$
\left\{\begin{array}{l}
\theta = e_{4} \\
\Omega =  \omega_{12}(e^{12}-e^{34}) + \omega_{14} e^{14}+ \omega_{23} e^{23}  + \omega_{24} e^{24}\\ 
\text{with }\omega_{12}^2-\omega_{14}\omega_{23} \neq 0
\end{array}\right. , $$
or
$$
\left\{\begin{array}{l}
 \theta = - e_{4}  \\
\Omega =  \omega_{12}(e^{12}+e^{34})+ \omega_{13} e^{13} + \omega_{14} e^{14}  + \omega_{24} e^{24}\\
\text{with }\omega_{12}^2+\omega_{13}\omega_{24} \neq 0
\end{array}\right. .$$

According to \cite{ovando} this Lie algebra admits two not equivalent complex structures defined as
$$
J_1=\left(\begin{array}{cccc}
0 & 0 & 1 & 0 \\
0 & 0 & 0 & 1 \\
-1 & 0 & 0 & 0 \\
0 & -1 & 0 & 0
\end{array}\right)
\in\mathrm{End}(\mathfrak{d}_4^\vee)
, \quad
J_2=\left(\begin{array}{cccc}
0 & 0 & 1 & 0 \\
1 & 0 & 0 & 1 \\
-1 & 0 & 0 & 0 \\
0 & -1 & -1 & 0
\end{array}\right)
\in\mathrm{End}(\mathfrak{d}_4^\vee).
$$

Let us start with $J_1$. Requiring $J_1$ to be positive we obtain $\omega_{13}<0$, in particular $\omega_{13}\neq0$. Then the only possibility for a compatible lcs structure is
$$
\left\{\begin{array}{l}
\theta = -e_{4} \\
\Omega =  \omega_{12}(e^{12}+e^{34})+ \omega_{13} e^{13} + \omega_{14} e^{14}  + \omega_{24} e^{24}\\
\text{with }\omega_{12}^2+\omega_{13}\omega_{24} \neq 0
\end{array}\right. .$$

Assuming $\Omega$ to be $J_1$-invariant and $J_1$-positive we obtain a generic lcK structure  
$$
\left\{\begin{array}{l}
	\theta = -e_{4} \\
	\Omega =  \omega_{12}(e^{12}+e^{34})+ \omega_{13} e^{13} + \omega_{24} e^{24}\\
	\text{with } \omega_{13}<0, \omega_{24}<0, -\omega_{12}^2+\omega_{13}\omega_{24} > 0
\end{array}\right. .$$

The generic automorphisms of ($\mathfrak{d}_4,J_1)$ in the chosen coframe are associated to
$$
\left(\begin{array}{cccc}
a_{11} & 0 & 0 & -a_{23} \\
0 & 1 & 0 & 0 \\
0 & a_{23} & a_{11} & 0 \\
0 & 0 & 0 & 1
\end{array}\right)
\in\mathrm{End}(\mathfrak{d}_4^\vee)
\text{ with }a_{11}\neq0.
$$
A generic automorphisim transforms $\theta=-e_{4}$ into $a_{23}e^1-e^4$ hence $a_{23}$ must be $0$. Then for $a_{11}=\frac{1}{\sqrt{-\omega_{13}}}$, we apply the automorphism
$$
\left(\begin{array}{cccc}
\frac{1}{\sqrt{-\omega_{13}}} & 0 & 0 & 0 \\
0 & 1 & 0 & 0 \\
0 & 0 & \frac{1}{\sqrt{-\omega_{13}}} & 0 \\
0 & 0 & 0 & 1
\end{array}\right)
\in\mathrm{End}(\mathfrak{d}_4^\vee)
$$
to get the lcK form
$$
\left\{\begin{array}{l}
\theta = - e^4 \\
\Omega = \mu(e^{12}+e^{34})- e^{1} \wedge e^{3} + \sigma  e^{2} \wedge e^{4}\\
\text{with } \mu^2+\sigma < 0
\end{array}\right. .$$
Finally applying the automorphism $a_{23}=0$ and $a_{21}=-1$ (if it is necessary) we can assume that $\mu\geq0$.
There is no further reduction, since the generic linear complex automorphism fixes the coefficient along $e^{24}$ and the sign of the coefficient along to $e^{12}$.

\begin{rmk}
The generic $A=a_1e_1+a_2e_2+a_3e_3+a_4e_4$ yields
$$ \mathrm{ad}_A = \left(\begin{matrix}
a_4 & 0 & 0 & -a_1 \\
0 & -a_4 & 0 & a_2 \\
-a_2 & a_1 & 0 & 0 \\
0 & 0 & 0 & 0
\end{matrix}\right). $$
Then $\mathrm{ad}_A$ is skew-symmetric if and only if $A=a_3 e_3$. But then $\theta(A)=0$. Therefore, by Lemma \ref{Vaisman}, there is no Vaisman structure among the above lcK structures.
\end{rmk}

Now we consider the complex structure $J_2$. We impose now $\Omega$ to be  $J_2$-positive. In the first and second case, there is no lcK structure because we need $\omega_{13}<0$. In the third case, we need $\omega_{12} - \omega_{24} > 0$ and $-\omega_{12} - \omega_{24} > 0$, but $J_2$-invariance for $\Omega$ yields $\omega_{24}=0$. Then, there is no lcK structure for $J_2$.

\subsection{$\mathfrak{d}_{4,1}$}

Consider the Lie algebra $\mathfrak{d}_{4,1} = (14,0,-12+34,0)$ with the complex structure 
$$
J = \left(\begin{array}{cccc}
0 & 0 & 0 & -1 \\
0 & 0 & -1 & 0 \\
0 & 1 & 0 & 0 \\
1 & 0 & 0 & 0
\end{array}\right)
\in\mathrm{End}(\mathfrak{d}_{4,1}^\vee)
$$
in terms of the chosen coframe.

According to \cite[Appendix 6.14 of the arXiv version, pages 61--62]{angella-bazzoni-parton} the generic lcs structures are either 
$$
\left\{\begin{array}{l}
\theta = \theta_4e^{4} \\
\Omega = \omega_{12} (e^{12}-(\theta_4+1)e^{34}) + \omega_{14} e^{14} +  \omega_{24} e^{24} \\
\text{with }\omega_{12}\neq 0, \theta_4\neq\{-1,-2,0\}
\end{array}\right. , $$
or
$$
\left\{\begin{array}{l}
\theta = -2e^{4} \\
\Omega = \omega_{12} (e^{12}+e^{34}) +\omega_{13} e^{13}+ \omega_{14} e^{14} + \omega_{24} e^{24}  \\
\text{with }\omega_{12}^2-\omega_{13}\omega_{24}\neq 0
\end{array}\right. , $$
or
$$
\left\{\begin{array}{l}
\theta = -e^{4} \\
\Omega = \omega_{12}e^{12} + \omega_{14} e^{14}+\omega_{23} e^{23} + \omega_{24} e^{24}  \\
\text{with }\omega_{14}\omega_{23}\neq 0
\end{array}\right. , $$
or
$$
\left\{\begin{array}{l}
\theta = \theta_{2}e^2+\theta_{4}e^{4} \\
\Omega = \omega_{12}e^{12} + \frac{(\theta_{4}+1)(\theta_{2}\omega_{12}-\omega_{23})}{\theta_{2}^2}e^{14} +\omega_{23} e^{23} + \omega_{24} e^{24} - \frac{(\theta_{4}+1)\omega_{23}}{\theta_{2}}  e^{34} \\
\text{with }(\theta_{4} + 1)\omega_{23}\neq 0, \theta_2 \neq 0
\end{array}\right. .$$

In the cases $\theta = \theta_4e^{4}$ with $\theta_{4}\neq -1$, we have that $\Omega$ is never $J$-positive: indeed $\Omega(e_2, Je_2)=\omega_{23}=0$. Then, there is no lcK structure in this case. 
In the case $\theta = -e^{4}$, the $J$-invariance of $\Omega$ requires $\omega_{12}=-\omega_{34}=0$ and $\omega_{13}=\omega_{24}=0$ and $J$-positive implies $\omega_{14}>0$ and $\omega_{23}>0$. Finally the generic lcK structure is 
$$
\left\{\begin{array}{l}
 \theta = -e_{4} \\
\Omega = \omega_{14}  e^{14} + \omega_{23} e^{23} \\
\text{with }\omega_{14}>0, \omega_{23}>0
\end{array}\right. .$$

The generic automorphisms of $(\mathfrak{d}_{4,1}, J)$ are associated to
\begin{equation}\label{aut_r4.1}
\left(\begin{array}{cccc}
1 & 0 & 0 & 0 \\
0 & a_{22} & 0 & 0 \\
0 & 0 & a_{22} & 0 \\
0 & 0 & 0 & 1
\end{array}\right)\in\mathrm{End}(\mathfrak{d}_{4,1}^\vee)
\text{ with }a_{22}^{2}\neq0.
\end{equation}
For $a_{22}=\frac{1}{\sqrt{\omega_{23}}}$, we apply the automorphism
$$
\left(\begin{array}{cccc}
1 & 0 & 0 & 0 \\
0 & \frac{1}{\sqrt{\omega_{23}}} & 0 & 0 \\
0 & 0 & \frac{1}{\sqrt{\omega_{23}}} & 0 \\
0 & 0 & 0 & 1
\end{array}\right)
\in\mathrm{End}(\mathfrak{d}_{4,1}^\vee)
$$
to get the lcK form
$$
\left\{\begin{array}{l}
\theta = - e^4 \\
\Omega = \sigma e^1\wedge e^4 + e^2\wedge e^3  \\
\text{with } \sigma > 0
\end{array}\right. .$$

There is no further reduction, since the generic linear complex automorphism transforms $ \sigma e^{14} + e^{23}$ into $ \sigma e^{14}+ a_{22}^{2} e^{23}$, fixing the coefficient along $e^{14}$.

Finally we consider the case 
$$
\left\{\begin{array}{l}
\theta = \theta_{2}e^2+\theta_{4}e^{4} \\
\Omega = \omega_{12}e^{12} + \frac{(\theta_{4}+1)(\theta_{2}\omega_{12}-\omega_{23})}{\theta_{2}^2}e^{14} +\omega_{23} e^{23} + \omega_{24} e^{24} - \frac{(\theta_{4}+1)\omega_{23}}{\theta_{2}}  e^{34} \\
\text{with }(\theta_{4} + 1)\omega_{23}\neq 0, \theta_2 \neq 0
\end{array}\right. .$$
Assuming $\Omega$ is $J$-invariant and $J$-positive, we reduced it to the lcK form
$$
\left\{\begin{array}{l}
\theta = \theta_{2}e^2+\theta_4 e^{4} \\
\Omega=\left( \frac{\omega_{23} {\left(\theta_{4} + 1\right)}}{\theta_{2}^{2}} \right)  e^{12} +  \left(\frac{\omega_{23}\theta_{4}(\theta_{4} + 1)}{\theta_{2}^{2}} \right)  e^{14} + \omega_{23} e^{23} + \left( -\frac{\omega_{23} {\left(\theta_{4} + 1\right)}}{\theta_{2}^{2}} \right)  e^{34} \\
\text{with } \omega_{23}> 0, \theta_4+1<0
\end{array}\right. .$$
We consider \eqref{aut_r4.1} with $a_{22}=\frac{1}{\theta_{2}}$ and we get
$$
\left\{\begin{array}{l}
\theta = e^2+\theta_4 e^{4} \\
\Omega=\frac{\omega_{23}}{\theta_{2}^{2}} \left( (\theta_{4} + 1)  ( e^{12} +  \theta_{4}e^{14}-e^{34} )+e^{23}  \right)\\
\text{with } \omega_{23}> 0, \theta_4+1<0
\end{array}\right. .$$
There is no further reduction because a generic automorphism applied to the Lee form $\theta= e^2+\theta_{4}e^{4}$ gives $ a_{22}e^2+\sigma e^{4}$, then the only possible automorphism between two lcK forms of this kind is the identity.

\subsection{$\mathfrak{d}_{4,\frac12}$}

Consider the Lie algebra $\mathfrak{d}_{4,\frac{1}{2}} = (\frac12 14,\frac12 24,-12+34,0)$. According to \cite[Appendix 6.14 of the arXiv version, pages 60--61]{angella-bazzoni-parton} the generic lcs structures are either 
$$
\left\{\begin{array}{l}
\theta = -\frac32 e^{4} \\
\Omega =  \omega_{12}(e^{12} + \frac12e^{34}) +  \omega_{13} e^{13}  +\omega_{14} e^{14}  + \omega_{23} e^{23}  + \omega_{24} e^{24}\\
\text{with } \omega_{12}^2-2\omega_{13}\omega_{24}+2\omega_{14}\omega_{23}\neq0 
\end{array}\right. , $$
or
$$
\left\{\begin{array}{l}
\theta = \theta_4 e^{4} \\
\Omega =  \omega_{12}(e^{12}-(\theta_{4}+1)e^{34}) + \omega_{14} e^{14} + \omega_{24} e^{24}\\ 
\text{with }\omega_{12}(\theta_{4}+1) \neq 0, \theta_{4}\neq -\frac32,  \theta_4 \neq 0 
\end{array}\right. .$$
According to \cite{ovando} $\mathfrak{d}_{4,\frac12}$ admits three different complex structures associated to
$$J_1=\left(\begin{array}{cccc}
	0 & -1 & 0 & 0 \\
	1 & 0 & 0 & 0 \\
	0 & 0 & 0 & 1 \\
	0 & 0 & -1 & 0
\end{array}\right)
\in\mathrm{End}(\mathfrak{d}_{4,\frac12}^\vee), \quad
J_2=\left(\begin{array}{cccc}
0 & 1 & 0 & 0 \\
-1 & 0 & 0 & 0 \\
0 & 0 & 0 & 1 \\
0 & 0 & -1 & 0
\end{array}\right)\in\mathrm{End}(\mathfrak{d}_{4,\frac12}^\vee), $$
$$
J_3=\left(\begin{array}{cccc}
0 & 0 & 0 & 1 \\
0 & 0 & 2 & 0 \\
0 & -\frac{1}{2} & 0 & 0 \\
-1 & 0 & 0 & 0
\end{array}\right)\in\mathrm{End}(\mathfrak{d}_{4,\frac12}^\vee).
$$
We consider first the complex structure $J_1$. If $\theta=-\frac32 e^4$, then the associated lcs form $\Omega$ is never $J_1$-positive. Indeed, $\Omega(e_1,J_1e_1)=\omega_{12}>0$ and $\Omega(e_3,J_1e_3)=-\omega_{34}=-\frac12\omega_{12}>0$, wich is a contradiction.
Therefore the only possibility is $\theta=\theta_{4} e^4$ with $\theta_{4}\neq -\frac32$. Assuming $\Omega$ is $J_1$-invariant and $J_1$-positive we reduce to
$$
\left\{\begin{array}{l}
\theta = \theta_4 e^{4} \\
\Omega =  \tau(e^{12}-(\sigma+1)e^{34})\\ 
\text{with }\tau>0, \sigma+1 >0, \theta_4 \neq 0
\end{array}\right. .$$
A generic automorphism for $(\mathfrak{d}_{4,\frac12},J_1)$ is given by
\begin{equation}\label{aut_d4.1/2}
\left(\begin{array}{cccc}
a_{11} & a_{12} & 0 & 0 \\
-a_{12} & a_{11} & 0 & 0 \\
0 & 0 & 1 & 0 \\
0 & 0 & 0 & 1
\end{array}\right)
\in\mathrm{End}(\mathfrak{d}_{4,\frac12}^\vee)
\text{ with } a_{11}^{2} + a_{12}^{2}=1.
\end{equation}
This automorphism fixes $\Omega$, therefore there is no further reduction in this case.

Now we focus on the complex structure $J_2$. If $\theta=-\frac32 e^4$, and we require $\Omega$ to be $J_2$-positive and $J_2$-invariant we obtain
$$
\left\{\begin{array}{l}
\theta = -\frac32 e^{4} \\
\Omega =  \omega_{12}  (e^{12}+\frac12e^{34}) + \omega_{13} (e^{13} + e^{24})  + \omega_{14} (e^{14}- e^{23})   \\ 
\text{with } \omega_{12}^2-2\omega_{13}^2-2\omega_{14}^2 > 0 , \omega_{12}<0
\end{array}\right. .$$
Suppose that $\omega_{13}=\omega_{14}=0$, then the lcK form is 
$$
\left\{\begin{array}{l}
\theta = -\frac32 e^{4} \\
\Omega =  \sigma  (e^{12}+\frac12e^{34})  \\ 
\text{with } \sigma<0. 
\end{array}\right. .$$
There is no further reduction since a generic automorphism for $(\mathfrak{d}_{4,\frac12},J_2)$ has the same form as in \eqref{aut_d4.1/2} fixing the coefficients of $\Omega$ along to $e_{12}$ and $e_{34}$.
If $\omega_{13}^2+\omega_{14}^2\neq0$, then applying \eqref{aut_d4.1/2} with 
$a_{11} = \frac{\omega_{13}} {\sqrt {\left( \omega_{13}^2+\omega_{14}^2 \right) }}$, 
$a_{12} =\frac{-\omega_{14}}{\sqrt{\left(\omega_{13}^2+\omega_{14}^2\right)}}$ we obtain the lcK form
$$
\left\{\begin{array}{l}
\theta = -\frac32 e^{4} \\
\Omega =  \sigma  (e^{12}+\frac12e^{34}) +\tau(e^{13}+e^{24}) \\ 
\text{with } \sigma<0, \tau>0. 
\end{array}\right. .$$
In the same way as above, there is no further reduction in this case.

Finally we consider the complex structure $J_3$. If $\theta=\theta_{4}e^4$ with $\theta_{4}\neq-\frac32$ then $\Omega$ is not positive. Indeed, $\Omega(e_2,J_3e_2)=0$. Therefore the only possibility for the Lee form is $\theta=-\frac32e^4$. Requiring $\Omega$ to be $J_3$-positive and $J_3$-invariant we get that $\omega_{12}=0$ and $2\omega_{13}=\omega_{24}$  and it reduces to the lcK form
$$
\left\{\begin{array}{l}
\theta = -\frac32 e^{4} \\
\Omega =  \omega_{13} (e^{13}+2e^{24}) +\omega_{14}e^{14}+\omega_{23}e^{23} \\ 
\text{with } \omega_{14}<0, \omega_{23}<0, \omega_{14}\omega_{23} - 2\omega_{13}^2>0
\end{array}\right. .$$
A generic automorphism for $(\mathfrak{d}_{4,\frac12},J_3)$ in the chosen coframe is 
$$
\left(\begin{array}{cccc}
1 & 0 & 0 & 0 \\
a_{12} & a_{22} & 0 & 0 \\
0 & 0 & a_{22} & 2 \, a_{12} \\
0 & 0 & 0	 & 1
\end{array}\right)\in\mathrm{End}(\mathfrak{d}_{4,\frac12}^\vee)
\text{ with } a_{22}\neq0.
$$
Applying the automorphism with $a_{12}=0$ and $a_{22}=\frac{1}{\sqrt{-\omega_{23}}}$ we get
$$
\left\{\begin{array}{l}
\theta = -\frac32 e^{4} \\
\Omega =  \mu (e^{13}+2e^{24}) +\sigma e^{14}-e^{23} \\ 
\text{with } \sigma<0, \sigma + 2\mu^2>0
\end{array}\right. ,$$
where we denote $\mu=\frac{\omega_{13}}{\sqrt{-\omega_{23}}}$ $\sigma=\omega_{14}$.
There is no further reduction, since a generic automorphism applied to the Lee form $\theta = -\frac32 e^{4}$ gives $ 2a_{12}e^3-\frac32 e^{4}$, then $a_{12}$ must be $0$. The only possible automorphisms between two lcK forms of this kind transform the coefficient along $e^{23}$ into $a_{22}^2e^{23}$, then $a_{22}=1$ and the automorphism is the identity.

\subsection{$\mathfrak{d}_{4,\lambda}$ with $\lambda>\frac12$, $\lambda\neq1$}

Consider the Lie algebra $\mathfrak{d}_{4,\lambda} = (\lambda 14,(1-\lambda) 24,-12+34,0)$ with $\lambda>\frac12$ and $\lambda\neq 1$. According to \cite[Appendix 6.14 of the arXiv version, page 58]{angella-bazzoni-parton} the generic lcs structures are either 
$$
\left\{\begin{array}{l}
\theta = -(1+\lambda) e^{4} \\
\Omega =  \omega_{12}(e^{12}+\lambda e^{34}) +  \omega_{13} e^{13}  +\omega_{14} e^{14} + \omega_{24} e^{24}\\
\text{with } \lambda\omega_{12}^2-\omega_{13}\omega_{24}\neq0 
\end{array}\right. ,$$
only when $\lambda \neq -1$,
or
$$
\left\{\begin{array}{l}
\theta = (\lambda-2) e^{4} \\
\Omega =  \omega_{12}(e^{12}-(\lambda-1)e^{34}) + \omega_{14} e^{14} + \omega_{23} e^{23} + \omega_{24} e^{24}\\ 
\text{with }(\lambda-1)\omega_{12}^2- \omega_{14}\omega_{23} \neq 0 \end{array}\right. ,$$
only when $\lambda \neq 2$,
or
$$
\left\{\begin{array}{l}
\theta = \theta_{4} e^{4}, \quad \text{with } \theta_{4}\neq -(1+\lambda), \lambda-2 \\
\Omega =  \omega_{12}(e^{12}-(\theta_{4}+1)e^{34}) + \omega_{14} e^{14} + \omega_{24} e^{24}\\ 
\text{with }\omega_{12}^2(\theta_{4}+1)\neq 0, \theta_4 \neq 0
\end{array}\right. .$$
According to \cite{ovando} $\mathfrak{d}_{4,\lambda}$ admits two different complex structures associated to
$$
J_1=\left(\begin{array}{cccc}
0 & 0 & 0 & -\frac{1}{\lambda} \\
0 & 0 & -1 & 0 \\
0 & 1 & 0 & 0 \\
{\lambda} & 0 & 0 & 0
\end{array}\right)
\in\mathrm{End}(\mathfrak{d}_{4,\lambda}^\vee),
\quad
J_2=\left(\begin{array}{cccc}
0 & 0 & -1 & 0 \\
0 & 0 & 0 & \frac{1}{1-\lambda} \\
1 & 0 & 0 & 0 \\
0 & {\lambda-1} & 0 & 0
\end{array}\right)
\in\mathrm{End}(\mathfrak{d}_{4,\lambda}^\vee).
$$
We consider first the complex structure $J_1$. Requiring $\Omega(e_2,J_1e_2)=\omega_{23}>0$, we get that the only possibility for the Lee form is $\theta=(\lambda-2)e^4$. Assuming $\Omega$ is $J_1$-invariant and $J_1$-positive we reduce to the following generic lcK structure
$$
\left\{\begin{array}{l}
\theta = (\lambda-2) e^{4} \\
\Omega =  \omega_{14}  e^{14} + \omega_{23}  e^{23}\\ 
\text{with }\omega_{14}>0, \omega_{23} >0, \lambda \neq 2
\end{array}\right..$$

A generic automorphism for $(\mathfrak{d}_{4,\lambda},J_1)$ is given by
$$
\left(\begin{array}{cccc}
1 & 0 & 0 & 0 \\
0 & a_{22} & 0 & 0 \\
0 & 0 & a_{22} & 0 \\
0 & 0 & 0 & 1
\end{array}\right)\in\mathrm{End}(\mathfrak{d}_{4,\lambda}^\vee)
\text{ with } a_{22}\neq0.
$$

We apply the automorphism with $a_{22}=\frac{1}{\sqrt{\omega_{23}}}$ and we reduce to lcK form
$$
\left\{\begin{array}{l}
\theta = (\lambda-2) e^{4} \\
\Omega =  \sigma  e^{14} + e^{23}\\ 
\text{with }\sigma>0, \lambda \neq 2
\end{array}\right.,$$
where $\sigma=\omega_{14}$. And there is no further reduction since a generic automorphism fixes the coefficient along to $e^{14}$.

Now we focus on the complex structure $J_2$. If $\theta\neq -(1+\lambda)e^4$, then associated $2$-form $\Omega$ is never positive. Indeed, $\Omega(e_1,J_2e_1)=\omega_{13}=0$. We consider the case $\theta= -(1+\lambda)e^4$. Assuming $\Omega$ is $J_2$-invariant and $J_2$-positive we obtain that $\omega_{12}=\omega_{14}=0$ and we reduce the lcK form to
$$
\left\{\begin{array}{l}
	\theta = -(1+\lambda) e^{4} \\
	\Omega =  \omega_{13}  e^{13} + \omega_{24}e^{24}\\ 
	\text{with }\omega_{13}>0, (\lambda - 1)\omega_{24}>0
\end{array}\right. ,
$$

A generic automorphism for $(\mathfrak{d}_{4,\lambda},J_2)$ is given by
$$
\left(\begin{array}{cccc}
a_{11} & 0 & 0 & 0 \\
0 & 1 & 0 & 0 \\
0 & 0 & a_{11} & 0 \\
0 & 0 & 0 & 1
\end{array}\right)
\in\mathrm{End}(\mathfrak{d}_{4,\lambda}^\vee)
\text{with } a_{11}\neq0 .
$$
Taking $a_{11}=\frac{1}{\sqrt{\omega_{13}}}$ we get the lcK form
$$
\left\{\begin{array}{l}
\theta = -(1+\lambda) e^{4} \\
\Omega =  e^{13} + \sigma e^{24}\\ 
\text{with }(\lambda - 1)\sigma>0
\end{array}\right. ,
$$
and there is no further reduction since a possible automorphism applied to $\Omega$ fixes the coefficient along to $e^{24}$.

\begin{rmk}
Consider all the cases $\mathfrak{d}_{4,\lambda}$ for $\lambda\geq\frac{1}{2}$ together. For the generic $A=a_1e_1+a_2e_2+a_3e_3+a_4e_4$, we get
$$ \mathrm{ad}_A = \left(\begin{matrix}
\lambda a_4 & 0 & 0 & -\lambda a_1 \\
0 & (a-\lambda) a_4 & 0 & (\lambda-1)a_2 \\
-a_2 & a_1 & a_4 & -a_3 \\
0 & 0 & 0 & 0 
\end{matrix}\right). $$
Therefore, $\mathrm{ad}_A$ is skew-symmetric if and only if $A=0$. By Lemma \ref{Vaisman}, we get that there is no Vaisman structure on $\mathfrak{d}_{4,\lambda}$ for any possible value of the parameter $\lambda$.
\end{rmk}

\subsection{$\mathfrak{d'}_{4,\delta}$ with $\delta\geq0$}

Consider the Lie algebra $\mathfrak{d'}_{4,\delta} = (\frac{\delta}{2}14+24,-14+\frac{\delta}{2}24,-12+\delta34,0)$ with $\delta\geq0$. According to \cite[Appendix 6.15 of the arXiv version, page 63]{angella-bazzoni-parton} the generic lcs structures are  
$$
\left\{\begin{array}{l}
\theta = \theta_{4} e^{4} \\
\Omega =  \omega_{12}(e^{12}-(\delta+\theta_{4})e^{34}) + \omega_{14} e^{14} + \omega_{24} e^{24}\\ 
\text{with }\omega_{12}^2(\delta+\theta_{4})\neq 0, \theta_4 \neq 0
\end{array}\right.. $$
According to \cite{ovando} $\mathfrak{d'}_{4,\delta}$  with $\delta\geq0$ admits two different complex structures associated to
$$
J_2=\left(\begin{array}{cccc}
0 & -1 & 0 & 0 \\
1 & 0 & 0 & 0 \\
0 & 0 & 0 & -1 \\
0 & 0 & 1 & 0
\end{array}\right)
\in\mathrm{End}((\mathfrak{d}^\prime_{4,\delta})^\vee), \quad
J_3=\left(\begin{array}{cccc}
0 & -1 & 0 & 0 \\
1 & 0 & 0 & 0 \\
0 & 0 & 0 & 1 \\
0 & 0 & -1 & 0
\end{array}\right)\in\mathrm{End}((\mathfrak{d}^\prime_{4,\delta})^\vee).
$$
In the case $\delta>0$ there are other two more non equivalent complex structures
$$
J_4=\left(\begin{array}{cccc}
0 & 1 & 0 & 0 \\
-1 & 0 & 0 & 0 \\
0 & 0 & 0 & -1 \\
0 & 0 & 1 & 0
\end{array}\right)\in\mathrm{End}((\mathfrak{d}^\prime_{4,\delta})^\vee), \quad
J_1=\left(\begin{array}{cccc}
0 & 1 & 0 & 0 \\
-1 & 0 & 0 & 0 \\
0 & 0 & 0 & 1 \\
0 & 0 & -1 & 0
\end{array}\right)\in\mathrm{End}((\mathfrak{d}^\prime_{4,\delta})^\vee).
$$
A generic automorphism for $(\mathfrak{d'}_{4,\delta},J)$ with $J\in\{J_1,J_2,J_3,J_4\}$ is given by
\begin{equation}\label{aut_d'4.delta}
\left(\begin{array}{cccc}
a_{11} & a_{12} & 0 & 0 \\
-a_{12} & a_{11} & 0 & 0 \\
0 & 0 & 1 & 0 \\
0 & 0 & 0 & 1
\end{array}\right)
\in\mathrm{End}((\mathfrak{d}^\prime_{4,\delta})^\vee)
\text{ with }a_{11}^{2} + a_{12}^{2}=1.
\end{equation}
Notice that for any choice of the complex structure the $J$-invariance condition implies that $\omega_{14}=0$ and $\omega_{24}=0$. Therefore the generic lcs structure reduces to
$$
\left\{\begin{array}{l}
\theta = \theta_{4} e^{4} \\
\Omega =  \omega_{12}(e^{12}-(\delta+\theta_{4})e^{34})\\ 
\text{with }\omega_{12}^2(\delta+\theta_{4})\neq 0, \theta_4 \neq 0
\end{array}\right., $$
and this lcs form is invariant by a generic automorphism given by \eqref{aut_d'4.delta}.

We consider first the complex structure $J_2$. Assuming $\Omega$ is $J_2$-invariant and $J_2$-positive we get
$$
\left\{\begin{array}{l}
\theta = \mu e^{4} \\
\Omega =  \sigma(e^{12}-(\delta+\mu)e^{34})\\ 
\text{with }\sigma>0, \delta+\mu< 0, \mu \neq 0
\end{array}\right. ,$$
where $\mu=\theta_{4}$ and $\sigma=\omega_{12}$. As we mention above there is no further reduction.

If we consider the complex structure $J_3$, in a very similar way we obtain
$$
\left\{\begin{array}{l}
\theta = \mu e^{4} \\
\Omega =  \sigma(e^{12}-(\delta+\mu)e^{34})\\ 
\text{with }\sigma>0, \delta+\mu> 0, \mu \neq 0
\end{array}\right. ,$$
where $\mu=\theta_{4}$ and $\sigma=\omega_{12}$.

We now focus on the complex structure $J_4$ (case $\delta>0$), and we have that the generic lcK structures are
$$
\left\{\begin{array}{l}
\theta = \mu e^{4} \\
\Omega =  \sigma(e^{12}-(\delta+\mu)e^{34})\\ 
\text{with }\sigma<0, \delta+\mu > 0, \mu \neq 0
\end{array}\right. ,$$
where $\mu=\theta_{4}$ and $\sigma=\omega_{12}$.

Finally, if $J=J_1$ (case $\delta>0$), then we obtain 
$$
\left\{\begin{array}{l}
\theta = \mu e^{4} \\
\Omega =  \sigma(e^{12}-(\delta+\mu)e^{34})\\ 
\text{with }\sigma<0, \delta+\mu < 0, \mu \neq 0
\end{array}\right. ,$$
where $\mu=\theta_{4}$ and $\sigma=\omega_{12}$.

\begin{rmk}
In any of the above four cases, it follows from Lemma \ref{Vaisman} that any lcK structure above is of Vaisman type if and only if $\delta=0$. Indeed, if $A\in \mathfrak{d'}_{4,\delta}$ such that  $\theta(A)=1$ and $A \in (\ker \theta)^\perp$, then $A=\frac{1}{\mu}e_4$. Therefore
$$ \operatorname{ad}_A = \left(\begin{matrix}
\frac{\delta}{2\mu} & \frac{1}{\mu} & 0 & 0 \\
-\frac{1}{\mu} & \frac{\delta}{2} & 0 & 0 \\
0 & 0 & \frac{\delta}{\mu} & 0 \\
0 & 0 & 0 & 0
\end{matrix}\right) $$
is skew-symmetric if and only if $\delta=0$. More precisely, all lcK structures on $\mathfrak{d'}_{4,0}$ are of Vaisman type, and any lcK structure on $\mathfrak{d'}_{4,\delta}$ with $\delta>0$ is not of Vaisman type. 
\end{rmk}

\subsection{$\mathfrak{h}_4$}

Consider the Lie algebra $\mathfrak{h}_4 = (\frac12 14 +24,\frac12 24,-12+34,0)$. According to \cite[Appendix 6.16 of the arXiv version, pages 64--65]{angella-bazzoni-parton} the generic lcs structures are either 
$$
\left\{\begin{array}{l}
\theta = \theta_{4} e^{4}\\
\Omega =  \omega_{12}(e^{12}-(\theta_{4}+1)e^{34}) + \omega_{14} e^{14} + \omega_{24} e^{24}\\ 
\text{with }\omega_{12}\neq 0, \theta_4 \not\in \{ -\frac{3}{2}, -1, 0 \}
\end{array}\right. ,$$
or
$$
\left\{\begin{array}{l}
\theta = -\frac32 e^{4} \\
\Omega =  \omega_{12}(e^{12}+\frac12e^{34}) + \omega_{14} e^{14} +  \omega_{23} e^{23}  + \omega_{24} e^{24}\\
\text{with } \omega_{12}^2+2\omega_{14}\omega_{23}\neq0
\end{array}\right. .$$
According to \cite{ovando}, $\mathfrak{h}_{4}$ admits a complex structure associated to
$$J=\left(\begin{array}{cccc}
0 & 0 & -2 & 0 \\
0 & 0 & 0 & 1 \\
\frac{1}{2} & 0 & 0 & 0 \\
0 & -1 & 0 & 0
\end{array}\right)
\in\mathrm{End}(\mathfrak{h}_{4}^\vee).
$$
In both cases we obtain that $\Omega$ is not $J$-positive, since $\Omega(e_1,Je_1)=0$. Therefore there is no lcK structure for this Lie algebra.

\section{Applications}
In this section, we show some applications of our classification of lcK structures in dimension $4$.
In particular, we adapt some constructions of lcs structures in \cite{origlia-construction} and \cite{angella-bazzoni-parton} to the lcK case and, as an application, we can produce many examples in higher dimension, including lcK structures on Oeljeklaus-Toma manifolds, or give a geometric interpretation of some of the $4$-dimensional structures in Table \ref{table:lck-solvable}.

\subsection{LcK extensions}
Let $\h$ be a Lie algebra equipped with an lcK structure $(J,\pint)$, and let $(\omega, \theta)$ be the underlying lcs structure. Let $V$ be a vector space of dimension $2n$ with a Hermitian structure $(J_0,\pint_0)$ and denote by $\omega_0$ the fundamental $2$-form induced by $(J_0,\pint_0)$.
We consider a representation
\[\pi: \h \to \operatorname{End}(V),\]
given by $\pi(X)=-\frac12\theta(X)\I+\rho(X)$ such that $\rho(X)\in\mathfrak{u}(n)\subset\mathfrak{sp}(n,\R)$ for all $X\in\h$.
According to \cite{origlia-construction}, the Lie algebra $\g$ defined by $\g=\h\ltimes_\pi V$ admits an lcs structure $(\omega', \theta')$ given by $\omega'|_\h=\omega$, $\omega'|_V=\omega_0$, $\omega'(X,Y)=0$ for any $X\in\h$, $Y\in V$ and the $1$-form $\theta'\in\g^*$ by $\theta'|_\h=\theta$ and $\theta'|_V=0$.

We define the almost Hermitian structure $(J', \pint')$ on $\g$ given by
\begin{equation}\label{eq:lck-extension}
\begin{array}{rcl}
\pint'|_\h=\pint,&\quad\quad& \pint'|_V=\pint_0,\\
J'|_\h=J, &\quad\quad& J'|_V=J_0.
\end{array}
\end{equation}
It is easy to see that $\omega'$ is the fundamental $2$-form associated to the almost Hermitian structure $(J', \pint')$ on $\g$.
Moreover, $J'$ is integrable since $\pi(X)\circ J_0=J_0\circ\pi(X)$ for any $X\in\h$ (see \cite{barberis-dotti}). Therefore, we obtain that:

\begin{prop}\label{construccionLCK}
Let $\h$ be a Lie algebra equipped with an lcK structure $(J,\pint)$ and let $V$ be a vector space endowed with a Hermitian structure $(J_0,\pint_0)$. Take the representation $\pi\colon \h \to \operatorname{End}(V)$
given by $\pi=-\frac12\theta\I+\rho$ where $\rho(X)\in\mathfrak{u}(n)\subset\mathfrak{sp}(n,\R)$. Then $(J', \pint')$ as in \eqref{eq:lck-extension} is an lcK structure on $\g$.
\end{prop}

\begin{rmk}\label{ro-cero}
	If the initial Lie algebra $\h$ is solvable, then $\rho(\h)$ is solvable. Since $\rho(\h)\subset\mathfrak{u}(n)$ and $\mathfrak{u}(n)$ is a compact Lie algebra, then we obtain that $\rho(\h)$ is Abelian, and therefore it is contained in a maximal Abelian subalgebra of $\mathfrak{u}(n)$. In particular $\rho(\h')=0$, where $\h'=[\h,\h]$ denotes the commutator ideal. Moreover, we may assume that $J\in\mathfrak{u}(n)$ is in the same maximal Abelian subalgebra which contains $\rho(\h)$.
\end{rmk}

According to \cite{origlia-construction}, we have that $\g$ is unimodular if and only if  $\tr(\ad_X^\h)=n\theta(X)$ for all $X\in\h$. Recall that given a Lie algebra $\h$, the map $\chi: \h\to \R$ defined by $\chi(X)=\operatorname{tr}(\ad_X)$ is a Lie algebra homomorphism, and its kernel is called the unimodular kernel of the Lie algebra $\h$. 
We have then the following result:

\begin{prop}\label{unimodularLCK}
Let $\h$ be a Lie algebra with an lcK structure $(\omega,\theta)$, and let $(\pi,V)$ be a $2n$-dimensional representation such that $\pi(X)=-\frac12\theta(X)\I+\rho(X)$ with $\rho(X)\in\mathfrak{u}(n)$ for all $X\in\h$. Then the Lie algebra $\g=\h\ltimes_\pi V$ with the lcK structure $(\omega',\theta')$  as above is unimodular if and only if the unimodular kernel of $\h$ is equal to  $\ker\theta$ and $\tr(\ad_A^\h)=n$.
\end{prop}

Note that in order to build unimodular examples we have to start with a non unimodular Lie algebra $\h$. In particular we need that   $\tr(\ad_A^\h)=n=\frac{\dim V}{2}\in\mathbb{N}$. This condition is enough when the commutator ideal $\h'=[\h,\h]=\ker\theta$.

Using Lemma \ref{Vaisman}, it is easy to see that the lcK structure constructed with Proposition \ref{construccionLCK} is not Vaisman. Indeed, the endomorphism $\ad_A^\g:\g\to\g$ is
\[\ad_A^\g=
\left(\begin{array}{c|c}
\ad_A^\h & \\
\hline
& -\frac12\I+\rho(A) \\
\end{array}\right),\]
with $\rho(A)$ skew-symmetric. Then $\ad_A^\g$ cannot be skew-symmetric, and therefore, the lcK structure is not Vaisman. Moreover, we will prove next that these Lie algebras do not admit any Vaisman structure when $n>1$:
 
\begin{prop}
Let $\g=\h\ltimes_\pi V$ be the unimodular solvable Lie algebra built as above with $\dim V=2n$ and $n>1$. Then $\g$ does not admit any Vaisman structure.
\end{prop}

\begin{proof}
	Suppose that $\g$ admits a Vaisman structure $(\omega,\theta)$. Then we know that the Morse-Novikov cohomology vanishes in any degree, and therefore $\omega=d_\theta\eta=d\eta-\theta\wedge\eta$ for some $1$-form $\eta$.
	Let $\tilde\omega$ the be restriction of $\omega$ to $V\times V$, then we have that $\tilde\omega$ is a symplectic form on $V$. Moreover, $\tilde\omega=-\theta\wedge\eta$ since $V$ is Abelian. If $n>1$, then $\tilde\omega$ is degenerate.
\end{proof}

\subsection{Examples arising from $4$-dimensional lcK Lie algebras}
We summarize which Lie algebras in Table \ref{table:lck-solvable} are not unimodular, and therefore can be used to construct new examples of unimodular Lie algebras of dimension higher or equal than six with an lcK structure. They are: $\mathfrak{rr}_{3,0}$, $\mathfrak{rr}_{3,1}$, $\mathfrak{rr'}_{3,\gamma}$, $\mathfrak{r}_2\mathfrak{r}_2$, $\mathfrak{r'}_2$, $\mathfrak{r}_{4,\alpha,1}$, $\mathfrak{r}_{4,\alpha,\alpha}$, $\mathfrak{r'}_{4,\gamma,\delta}$, $\mathfrak{d}_{4,\lambda}$, and $\mathfrak{d'}_{4,\delta}$. 

Taking into account the lcK structures for each of these Lie algebras as exhibited in Table \ref{table:lck-solvable}, it can be shown, using Proposition \ref{unimodularLCK}, that the only lcK Lie algebras which admit a unimodular lcK extension are the following, where the complex structure is written in the frame $\{e_1,e_2,e_3,e_4\}$:

\
$\begin{array}{l} \mathfrak{r}_2\mathfrak{r}_2: \quad \quad \quad \; \\ 
%(0,-12,0,-34)
\end{array}
\left\{\begin{array}{l}
\theta = \sigma (e^1+ e^3) \\
\omega =  \mu ( \frac{1+\sigma}{\sigma} e^{12} + e^{14} - e^{23} + \frac{\sigma+1}{\sigma} e^{34})  \\ \sigma\neq 0, \sigma>-\frac12, \mu\neq0, \frac{\mu(1+\sigma)}{\sigma}>0
\end{array}\right. \quad
{\tiny J = \left(\begin{array}{cccc}
0 & -1 & 0 & 0 \\
1 & 0 & 0 & 0 \\
0 & 0 & 0 & -1 \\
0 & 0 & 1 & 0
\end{array}\right)} $

\

\

$\begin{array}{l} \mathfrak{r}_{4,\alpha,\alpha}:\\ %(14,\alpha 24,\alpha 34,0)\\  
\alpha\not\in\{-1,0,1\} \end{array}  
\left\{\begin{array}{l}
\theta = -2\alpha e^4\\
\omega = \sigma e^{14}+e^{23}  \\
\sigma<0
\end{array}\right. \quad \quad \quad \quad \quad \quad \quad\quad\quad \;
{\tiny J= \left(\begin{array}{cccc}
0 & 0 & 0 & 1 \\
0 & 0 & -1 & 0 \\
0 & 1 & 0 & 0 \\
-1 & 0 & 0 & 0
\end{array}\right)} $

\

\

$\begin{array}{l} \mathfrak{r}^\prime_{4,\gamma,\delta}: \;\; \quad \quad \quad\\ %(14,\gamma 24+\delta 34,-\delta 24+\gamma 34,0)\\
\delta>0, \gamma\neq0\end{array} 
\left\{\begin{array}{l}
\theta = -2\gamma e^4 \\
\omega =\sigma e^{14}+e^{23}  \\
 \sigma<0
\end{array}\right. \quad \quad \quad \quad \quad \quad \quad \quad \quad \;
{\tiny J = \left(\begin{array}{cccc}
0 & 0 & 0 & 1 \\
0 & 0 & -1 & 0 \\
0 & 1 & 0 & 0 \\
-1 & 0 & 0 & 0
\end{array}\right)}$

\

\

$\begin{array}{l}\mathfrak{d}_{4,\frac{1}{2}}: \quad \quad \quad\quad \\ %(\frac12\times 14,\frac12\times 24,-12+34,0) 
\end{array}  
\left\{\begin{array}{l}
\theta = \theta_4e^4 \\
\omega =\tau (e^{12}-(\sigma+1)e^{34}))  \\
 \tau>0, \sigma+1>0, \theta_4\neq0
\end{array}\right. \quad \quad \quad \quad \quad 
{\tiny J=\left(\begin{array}{cccc}
	0 & -1 & 0 & 0 \\
	1 & 0 & 0 & 0 \\
	0 & 0 & 0 & 1 \\
	0 & 0 & -1 & 0
\end{array}\right)}$

\

\

$\begin{array}{l}\mathfrak{d}_{4,\lambda}:\quad \quad \quad \quad \,\\ %(\lambda 14,(1-\lambda)24,-12+34,0)\\
\lambda\not\in\{\frac{1}{2},1\} \end{array} 
\left\{\begin{array}{l}
\theta = (\lambda-2)e^4 \\
\omega =\sigma e^{14}+e^{23}  \\
 \sigma>0,  \lambda \neq 2
\end{array}\right. \quad  \quad \quad \quad \quad \quad \quad \quad \quad
{\tiny J=\left(\begin{array}{cccc}
0 & 0 & 0 & -{\lambda}  \\
0 & 0 & -1 & 0 \\
0 & 1 & 0 & 0 \\
\frac{1}{\lambda}& 0 & 0 & 0
\end{array}\right)} $

\

\

$\begin{array}{l}\mathfrak{d}^\prime_{4,\delta}: \quad \quad \quad\quad \;\\ %{\tiny(\frac{\delta}{2}14+24,-14+\frac{\delta}{2}24,-12+\delta 34,0)}\\ 
\delta>0 \end{array}  
\left\{\begin{array}{l}
\theta =\mu e^4 \\
\omega =\sigma (e^{12}-(\delta+\mu) e^{34})  \\
\sigma<0, \delta+\mu<0, \mu\neq0
\end{array}\right. \quad \quad \quad \quad \quad
{\tiny J=\left(\begin{array}{cccc}
0 & -1 & 0 & 0 \\
1 & 0 & 0 & 0 \\
0 & 0 & 0 & 1 \\
0 & 0 & -1 & 0
\end{array}\right) }$

\begin{rmk}
Applying Proposition \ref{construccionLCK} to the Lie algebras $\mathfrak{r}_{4,\alpha,\alpha}$ and $\mathfrak{r}^\prime_{4,\gamma,\delta}$, we obtain Lie algebras of dimension greater than or equal to six, which are almost-Abelian Lie algebras. Recall that a Lie algebra is called almost-Abelian if it has an Abelian ideal of codimension one. Almost-Abelian Lie algebras admitting lcK structures were studied in \cite{andrada-origlia-MM}, where the second-named author and Adri\'an Andrada proved that the associated Lie groups admit no lattices, whenever the dimension is greater than four.
\end{rmk}

\begin{rmk}
If we extend the Lie algebras $\mathfrak{d}_{4,\frac{1}{2}}$, $\mathfrak{d}_{4,\lambda}$ and $\mathfrak{d}^\prime_{4,\delta}$ by Proposition \ref{construccionLCK}, we obtain almost-nilpotent Lie algebras of dimension greater than or equal to six with an lcK structure. We will explain in detail how to extend one of these Lie algebras in Example \ref{ex:ext-alm-nilp}. Recall that a Lie algebra is called almost-nilpotent if it has a nilpotent ideal of codimension one. The existence of lattices in almost-nilpotent Lie groups was studied in \cite{bock}.
\end{rmk}

We explain now how to extend one of the almost nilpotent cases:
\begin{exa}\label{ex:ext-alm-nilp}
Let us start with the $4$-dimensional Lie algebra $\mathfrak d'_{4,\delta}$ with $\delta\neq 0$ and structure constants $(\frac {\delta}{2}14 +24, -14+\frac{\delta}{2}24,-12+\delta34,0)$.
Then the non-vanishing Lie brackets are given by 
\[\displaystyle{[e_4,e_1]=\frac {\delta}{2}e_1-e_2,\quad [e_4,e_2]=e_1+\frac {\delta}{2}e_2,\quad [e_4,e_3]=\delta e_3,\quad [e_1,e_2]=e_3}.\]
We consider the lcK structure on $\mathfrak d'_{4,\delta}$ given by 
$$\left\{\begin{array}{l}
\theta =\mu e^4 \\
\omega =\sigma (e^{12}-(\delta+\mu) e^{34})  \\
\sigma<0, \delta+\mu<0, \mu\neq0
\end{array}\right., \quad  
J=\left(\begin{array}{cccc}
0 & -1 & 0 & 0 \\
1 & 0 & 0 & 0 \\
0 & 0 & 0 & 1 \\
0 & 0 & -1 & 0
\end{array}\right). $$ 

We show that, for any $n\in \mathbb N$, there is an lcK extension given by Proposition \ref{construccionLCK} of the Lie algebra $\mathfrak d'_{4,\delta}$ for a suitable choice of $\delta$ and $\mu$ in order to obtain a  $(2n+4)$-dimensional unimodular lcK Lie algebra $\g=\mathfrak d'_{4,\delta}\ltimes_\pi \R^{2n}$ for certain lcK representation $\pi$. It follows from Proposition \ref{unimodularLCK} that $\g$ is unimodular if and only if $2\delta=n\mu$. 

We define next the representation $\pi: \mathfrak d'_{4,\delta} \to \mathfrak{gl}(2n,\R)$ by  $\pi=-\frac12\theta \I+\rho$ for some representation $\rho:\mathfrak d'_{4,\delta}\to \mathfrak u(n)$. It follows from Remark \ref{ro-cero} that such a representation satisfies $\rho(e_1)=\rho(e_2)=\rho(e_3)=0$. 
Setting $\rho(e_4)$ in the orthonormal basis $\{u_1,v_1,\dots,u_n,v_n\}$ of $\R^{2n}$ given by 
$$\rho(e_4)=\begin{pmatrix}
	0 & a_1   && & \\
   -a_1 & 0   && & \\
	&& \ddots && \\
	&& &  0 & a_n  \\
	&& &  -a_n & 0   \\
\end{pmatrix}$$
we obtain that the only non-zero new Lie brackets are 
$$[e_4,u_i]=-\frac{\delta}{n}u_i-a_i v_i, \quad  [e_4,v_i]=a_i u_i -\frac{\delta}{n}v_i,$$
for $i=1,\dots,n$. The complex structure on $\g$ is given by $J(e_1)=e_2$, $J(e_3)=-e_4$ and $J(u_i)=v_i$ for $i=1,\dots,n$.

It is easy to verify that 
$$\left\{\begin{array}{l}
\theta' =\frac{2\delta}{n} e^4 \\
 {\displaystyle \omega' =\sigma (e^{12}-\frac{(n+2)\delta}{n} e^{34})  + \sum_{i=1}^n u^i\wedge v^i}\\
\text{ with} \quad \sigma<0, \delta<0
\end{array}\right.$$
is an lcK structure on the $(2n+4)$-dimensional Lie algebra $\g$ for any $n\in \mathbb N$. Note that we can write $\g=\R e_4\ltimes(\h_3\times\R^{2n})$, therefore $\g$ is an almost-nilpotent Lie algebra.
\end{exa}

\subsection{OT Lie algebras as lcK extensions }
Oeljeklaus-Toma manifolds (OT manifolds) are compact complex non-K\"aher manifolds which arise from certain number fields, and they can be considered as generalizations of the Inoue sufaces of type $S^0$. It was proved in \cite{oeljeklaus-toma} that certain OT manifolds (those of type $(s,1)$) admit lcK metrics. According to \cite{kasuya}, the OT manifolds are solvmanifolds.
Moreover, it can be seen that the complex structure is induced by a left-invariant one on the corresponding simply connected solvable Lie group.
These manifolds provided a counterexample to a conjecture made by Vaisman according to which the first Betti number of a compact lcK manifold is odd (see \cite{oeljeklaus-toma}).
We show in this subsection that the Lie algebras associated to these OT solvmanifolds of type $(s,1)$ endowed with its lcK structure can be obtained using our construction given in Proposition \ref{construccionLCK}.

We recall the definition of the Lie algebra associated to the $(2n+2)$-dimensional Oeljeklaus-Toma solvmanifold of type $(s,1)$ (see \cite{kasuya}), which we denote by $\g_{\text{OT}}$. The Lie brackets on $\g_{\text{OT}}$ are given by 
\begin{equation}\label{OT-LieBracket}
[x_i,y_i]=y_i, \quad [x_i,z_1]=-\frac12 z_1 + c_i z_2, \quad [x_i,z_2]= -c_i z_1 -\frac12 z_2,
\end{equation} 
in the basis $\{x_1,\dots,x_n,y_1,\dots,y_n,z_1,z_2\}$ for some $c_i\in\R$. The complex structure on $\g_{\text{OT}}$ is $Jx_i=y_i$ and $Jz_1=z_2$. 
The lcK structure on $\g_{\text{OT}}$ is given by
\begin{equation} \label{LCK_OT}
\left\{\begin{array}{l}
\theta = \sum x^i\\
\omega = 2\sum x^i\wedge y^i  + \sum_{i\neq j}x^i\wedge y^j + z^1\wedge z^2
\end{array}\right. 
\end{equation}
in the dual basis $\{x^1,\dots,x^n,y^1,\dots,y^n,z^1,z^2\}$,
where the associated Hermitian metric is defined by $g(\cdot ,\cdot)=\omega(\cdot ,J\cdot)$.

\ 

Next we show how to recover the $6$-dimensional OT Lie algebra using Proposition \ref{construccionLCK} and Theorem \ref{thm:classification-lck}.
Let us consider the $4$-dimensional Lie algebra $\mathfrak r_2\mathfrak r_2$ with structure constants given by $(0,-12,0,-34)$. 

According to Table \ref{table:lck-solvable} this Lie algebra admits many non equivalent LCK structures up to Lie algebra complex automorphisms. The one we are interested now is 
$$ \left\{\begin{array}{l}
\theta = \sigma e^1 +\tau e^3\\
\omega = \frac{1+\sigma}{\tau}e^{12} + e^{14} -e^{23} + \frac{\tau+1}{\sigma} e^{34}  \\
\text{	with }\sigma\tau\neq0, \sigma+\tau\neq -1,  \frac{1+\sigma}{\tau}>0, \frac{\tau+1}{\sigma}>0, \sigma\leq\tau
\end{array}\right.. $$ The complex structure is given by $J(e^1)=e^2$ and $J(e^3)=e^4$ in terms of the coframe $\{e^1,e^2,e^3,e^4\}$.

Taking into account Proposition \ref{unimodularLCK} in order to obtain a $6$-dimensional unimodular extension of this Lie algebra, we can simplify the lcK form to
$$ \left\{\begin{array}{l}
\theta = e^1 +e^3\\
\omega = 2e^{12} + e^{14} -e^{23} + 2e^{34}
\end{array}\right.. $$

Let $\g$ be the $6$-dimensional vector space $\g=\mathfrak r_2\mathfrak r_2\oplus \R^{2}.$ We define $\pi\colon \mathfrak r_2\mathfrak r_2 \to \mathfrak{gl}(2,\R)$ by
\begin{eqnarray*}
\pi(e_2)&=&0, \\
\pi(e_4)&=&0, \\
\pi(e_1) &=& -\frac12\I +
\begin{pmatrix}
0 & -c_1 \\
c_1 & 0 
\end{pmatrix}=\begin{pmatrix}
-\frac12 & -c_1 \\
c_1 & -\frac12 
\end{pmatrix},\\
\pi(e_3)&=& -\frac12\I +
\begin{pmatrix}
0 & -c_2 \\
c_2 & 0 
\end{pmatrix}=\begin{pmatrix}
-\frac12 & -c_2 \\
c_2 & -\frac12 
\end{pmatrix},
\end{eqnarray*}
in the orthonormal basis $\{e_5,e_6\}$ of $\R^2$ with $J_0e_5=e_6$. It is easy to check that $\pi$ satisfies conditions of Proposition \ref{construccionLCK}, and therefore  
\[\g=\mathfrak r_2\mathfrak r_2\ltimes_\pi \R^{2}\]
is a $6$-dimensional unimodular Lie algebra admitting a lcK structure $(\omega,\theta)$ given by 	
$$ \left\{\begin{array}{l}
\theta' = e^1 +e^3\\
\omega' = 2e^{12} + e^{14} -e^{23} + 2e^{34} + e^{56} 
\end{array}\right.. $$
Clearly, $(\g,\omega',\theta')$ is isomorphic to the OT Lie algebra of dimension $6$ with lcK structure given by \eqref{LCK_OT} with $n=2$.

\smallskip

To generalize this case to higher dimensions we consider the non-unimodular $2n$-dimensional Lie algebra $\mathfrak{aff}(\R)^n$ with structure equations given by $ [e_i,f_i]=f_i$ for $i=1,\dots,n$ in the basis $\{e_1,f_1,\dots,e_n,f_n\}$. This Lie algebra admits a complex structure $Je_i=f_i$ and $Jf_i=-e_i$ for $i=1,\dots,n$. A lcK structure on $\mathfrak{aff}(\R)^n$ can be defined by 
$$
\left\{\begin{array}{l}
\theta = \sum e^i\\
\omega = 2\sum e^i\wedge f^i  + \sum_{i\neq j}e^i\wedge f^j
\end{array}\right. 
$$
in the dual basis $\{e^1,f^1,\dots,e^n,f^n\}$,
with Hermitian metric $g(\cdot ,\cdot)=\omega(\cdot ,J\cdot)$.
Let $\g$ be the $(2n+2)$-dimensional vector space $\g=\mathfrak{aff}(\R)^n\oplus \R^{2}$. We extend the complex structure $J$ in $\mathfrak{aff}(\R)^n$ to $\g$ by $Ju_1=u_2$ where $\{u_1,u_2\}$ denotes an orthonormal basis of $\R^{2}$.
We define $\pi:\mathfrak{aff}(\R)^n \to \mathfrak{gl}(2,\R)$ by
$$
\pi(e_i)= 
-\frac12\I +
\begin{pmatrix}

0 & -c_i \\
c_i & 0 
\end{pmatrix}
=\begin{pmatrix}

-\frac12 & -c_i \\
c_i & -\frac12 
\end{pmatrix},
$$
and  $\pi(f_i)=0$ for $i=1,\dots,n$ and $c_i\in\R$. It is easy to check that $\pi$ satisfies Proposition \ref{construccionLCK}, and therefore the unimodular Lie algebra
\[\g=\mathfrak r_2\mathfrak r_2\ltimes_\pi \R^{2n},\]
admits a lcK structure, which we still denote by $(\omega, \theta)$, given by
$$
\left\{\begin{array}{l}
\theta = \sum e^i\\
\omega = 2\sum e^i\wedge f^i  + \sum_{i\neq j}e^i\wedge f^j + u^1\wedge u^2
\end{array}\right. .
$$

It is clear that $\g$ is isomorphic to the $(2n+2)$-dimensional Oeljeklaus-Toma Lie algebra with Lie bracket given by \eqref{OT-LieBracket}. Indeed, $\phi\colon \g\to \g_{\text{OT}}$, $\phi(e_i)=x_i$, $\phi(f_i)=y_i$ for $i=1,\dots,n$ and $\phi(u_i)=z_i$ for $i=1,2$ is a Lie algebra isomorphism which commutes with the complex structures and it also preserves the lcK forms.

\subsection{From coK\"ahler to lcK}
In \cite{angella-bazzoni-parton}, the first-named author and G. Bazzoni and M. Parton observed that every lcs structure on $4$-dimensional Lie algebras can be constructed either as a solution to the cotangent extension problem \cite[Corollary 1.14]{angella-bazzoni-parton}, or as a mapping torus over a contact $3$-dimensional Lie algebra \cite[Theorem 1.4]{angella-bazzoni-parton}, or with a similar construction starting from a $3$-dimensional cosymplectic Lie algebra \cite[Proposition 1.8]{angella-bazzoni-parton}.

We recall that, on a $(2n-1)$-dimensional Lie algebra $\mathfrak h$, an {\em almost-contact metric structure} is given by $(\eta,\xi,\Phi,g)$ where: $\eta$ is a $1$-form and $\xi$ a vector field such that
\begin{equation}\tag{cK1}
\eta(\xi)=1,
\end{equation}
$\Phi\in\mathrm{End}_{\mathbb R}(\mathfrak h)$ satisfies
\begin{equation}\tag{cK2}
\Phi^2=-\mathrm{id}+\eta\otimes\xi,
\end{equation}
and $g$ is a Riemannian metric such that
\begin{equation}\tag{cK3}
g(\Phi x,\Phi y)=g(x,y)-\eta(x)\eta(y)
\end{equation}
for any $x,y\in\mathfrak h$. Set $\omega:=g(\_,\Phi\_)$. If
\begin{equation}\tag{cK4}
d\eta=d\omega=0,
\end{equation}
then in particular $(\eta,\omega)$ is a {\em cosymplectic structure}.
For a cosymplectic structure, we denote by $R$ the Reeb vector, determined by $\iota_R\omega=0$ and $\iota_R\eta=1$; then one has the decomposition $\mathfrak h^*=\langle \eta \rangle \oplus \langle R \rangle^\circ$, where $\langle R \rangle^\circ$ denotes the annihilator of $\langle R \rangle$, that coincides with the kernel of the map $\omega^{n-1}\wedge\_$.
Recall also that $(\eta,\xi,\Phi)$ is called {\em normal} when
\begin{equation}\tag{cK5}
\mathrm{Nij}_\Phi+2d\eta\otimes\xi=0.
\end{equation}
When $(\eta,\xi,\Phi,g)$ is both cosymplectic and normal, then it is called {\em coK\"ahler}. We refer to {\itshape e.g.} \cite{blair} for further details.

\begin{prop}\label{cosymp_lcs}
Let $(\mathfrak h,\eta, \xi, \Phi, g)$ be a coK\"ahler Lie algebra of dimension $2n-1$, endowed with a derivation $D$ such that $D\omega=\alpha\omega$ for some $\alpha\neq 0$, $D\eta=D\xi=0$ , and $D\Phi=\Phi D$. Then $\mathfrak g=\mathfrak h\rtimes_D\mathbb R$ admits a natural lcK structure.
\end{prop}

\begin{proof}
By \cite[Proposition 1.8]{angella-bazzoni-parton}, we already know that $\mathfrak g$ has a natural lcs structure. For the sake of completeness, we briefly recall the construction. On $\mathfrak g = \mathfrak h \rtimes_D \mathbb R$, we define
$$ \theta(X,a):=-\alpha a, \qquad \Omega:=\omega+\eta\wedge\theta .$$
Therefore one has
$$ d^{\mathfrak{g}}\Omega = d^{\mathfrak{h}}\omega + \frac{1}{\alpha} D^*\omega \wedge \theta = \omega\wedge\theta = \theta\wedge\Omega . $$

It suffices to show that $\Omega$ is actually lcK, that is, there is a natural integrable complex structure $J$ on $\mathfrak g$ such that $J\Omega=\Omega$. We set, see \cite{sasaki-hatakeyama}, see also \cite[Section 6.1]{blair},
$$ J(X,a) := \left( \Phi X-a\xi, \eta(X) \right) . $$

We recall that $\Phi\xi=0$ and $\eta\circ\Phi=0$, see {\itshape e.g.} \cite[Theorem 4.1]{blair}, whence $J^2=-\mathrm{id}$.

We claim that $\mathrm{Nij}_J=0$. Recall that $\mathrm{Nij}_J:=-[\_,\_]+[J\_,J\_]-J[J\_,\_]-J[\_,J\_]$. Following the computations in \cite[Section 6.1]{blair}, we compute
\begin{eqnarray*}
	\mathrm{Nij}_J((X,0),(Y,0)) &=& \left( \mathrm{Nij}_\Phi(X,Y) + \eta(X)D\Phi Y-\eta(Y)D\Phi X -\eta(X)\Phi D Y+\eta(Y) \Phi DX, \right. \\
	&& \left. - \eta[\Phi X,Y]-\eta[X,\Phi Y]-\eta(X)\eta(DY)+\eta(Y)\eta(DX) \right) \\
	&=& \left( \mathrm{Nij}_\Phi(X,Y), 
	 (d\eta)(\Phi X,Y)+(d\eta)(X,\Phi Y)-\eta\wedge D\eta(X,Y) \right) \\
	&=& (0,-\eta\wedge D\eta(X,Y)) = 0 ,\\
	\mathrm{Nij}_J((X,0),(0,a)) &=& \left( -aDX-a[\Phi X,\xi]-a\eta(X)D\xi-a\Phi (D(\Phi X))+a\Phi[X,\xi], \right. \\
	&& \left. -a\eta(D(\Phi X))-a\eta([X,\xi]) \right)\\
	&=& \left( -aDX-a[\Phi X,\xi]-a\Phi^2 (DX)+a\Phi[X,\xi], \right. \\
	&& \left. -a\eta\circ\Phi(D X)+a (d\eta)(X,\xi) \right) \\
	&=& (-a[\Phi X,\xi]-a\eta(DX)\xi +a\Phi[X,\xi],0)=0,
\end{eqnarray*}
where we use that $\mathrm{Nij}_{\Phi}=0$.

We claim that $J\Omega=\Omega$. Indeed we compute
\begin{eqnarray*}
\Omega(J(X,a), J(X,b))
&=& (\omega+\eta\wedge\theta)((\Phi X-a\xi,\eta(X)), (\Phi Y-b\xi,\eta(Y))) \\
&=& \omega(\Phi X-a\xi,\Phi Y-b\xi)-\alpha\eta(\Phi X-a\xi)\eta(Y)+\alpha\eta(\Phi Y-b\xi)\eta(X) \\
&=& \omega(\Phi X,\Phi Y)-b\omega(\Phi X,\xi)-a\omega(\xi,\Phi Y)+ab\omega(\xi,\xi)\\
&& -\alpha\eta(\Phi X)\eta(Y)+a\alpha\eta(\xi)\eta(Y)+\alpha\eta(\Phi Y)\eta(X)-b\alpha\eta(\xi)\eta(X) \\
&=& \omega(X,Y)-\eta(Y)\omega(X,\xi)-b\omega(\Phi X,\xi)-a\omega(\xi,\Phi Y)+ab\omega(\xi,\xi)\\
&& -\alpha\eta(\Phi X)\eta(Y)+a\alpha\eta(\xi)\eta(Y)+\alpha\eta(\Phi Y)\eta(X)-b\alpha\eta(\xi)\eta(X) \\
&=& \omega(X,Y)-\alpha\eta(\Phi X)\eta(Y)+a\alpha\eta(Y)+\alpha\eta(\Phi Y)\eta(X)-b\alpha\eta(X) \\
&=& \omega(X,Y)+a\alpha\eta(Y)-b\alpha\eta(X) \\
&=& \Omega((X,a),(Y,b)),
\end{eqnarray*}
where we used that $\omega(\Phi X, \Phi Y)=g(\Phi X,\Phi^2 Y)=g(\Phi X,-Y+\eta(Y)\xi)=-g(\Phi X,Y)+\eta(Y)g(\Phi X,\xi)=-\omega(Y,X)+\eta(Y)\omega(\xi,X)=\omega(X,Y)-\eta(Y)\omega(X,\xi)$ and $\omega(\xi,Z)=-\omega(Z,\xi)=-g(Z,\Phi\xi)=0$.
\end{proof}

\begin{rmk}
The Lie algebra $\mathfrak g$ is unimodular if and only if $\mathfrak h$ is unimodular and $D\eta=-\alpha(n-1)\eta+\zeta$ for some $\zeta\in\langle R\rangle^\circ$. If $\mathfrak h$ is unimodular then the lcK structure $(\Omega,\vartheta)$ on $\mathfrak g$ is not exact.

Indeed, we recall the idea in \cite{angella-bazzoni-parton}: unimodularity for $\mathfrak{g}$ is equivalent to the generator of $\wedge^n\mathfrak g^*$ being non-exact, that is equivalent to $d\wedge^{n-1}\mathfrak g^*=0$. Since $\wedge^{n-1}\mathfrak g^*=\langle \omega^{n-1}\wedge\eta \rangle \oplus \wedge^{2n-2}\mathfrak h^*\wedge \theta$, we compute
$$ d^{\mathfrak{g}}(\omega^{n-1}\wedge \eta)=-\frac{1}{\alpha}\left( \alpha(n-1)+\beta\right) \omega^{n-1}\wedge\eta\wedge\theta, \qquad d^{\mathfrak{g}}(\varphi\wedge\theta)=d^{\mathfrak{h}}\varphi\wedge\theta, $$
where $\varphi\in\wedge^{2n-2}\mathfrak h^*$, and we decomposed $D\eta=\beta\eta+\zeta$ with $\zeta\in\langle R \rangle^\circ$. The statement follows.
\end{rmk}

Next we show an example of a Lie algebra admitting a lcK structure in Table \ref{table:lck-solvable} constructed from a $3$-dimensional coK\"ahler Lie algebra.
Recall from \cite{fino-vezzoni} that coK\"ahler Lie algebras in dimension $2n + 1$ are in one-to-one correspondence with $2n$-dimensional K\"ahler Lie algebras endowed with a skew-adjoint derivation $B$ which commutes with
its complex structure.

Let $(\R^2, e^1\wedge e^2)$ be a $2$-dimensional K\"ahler Lie algebra, where $Je^1=e^2$ in the orthonormal coframe $\{e^1,e^2\}$. Consider the derivation of $\R^2$ given by $B(e^1)=e^2$ and $B(e^2)=-e^1$.
Then the Lie algebra $\h=\R^2\rtimes_B \R\xi$ admits a coK\"ahler structures $(\eta, \xi, \Phi, g)$ where $g$ is the orthonormal extension of the K\"ahler metric in $\R^2$, $\eta$ is the dual $1$-form of $\xi$ and $\Phi=J$ in $\R^2$ and $\Phi(\xi)=0$.

Let us consider now the derivation $D: \h \to \h$ given by $D(\xi)=0$, $D(e^1)=e^2$ and $D(e^2)=-e^1$.
Finally, the Lie algebra $\mathfrak g=\mathfrak h\rtimes_D\mathbb R$ admits a natural lcK structure according Proposition \ref{cosymp_lcs}, and it is easy to see that this Lie algebra is isomorphic to the $4$-dimensional Lie algebra $\mathfrak r'_2$ on Table \ref{table:lck-solvable}.

\begin{rmk}
Concerning the construction in \cite{angella-bazzoni-parton} as a mapping torus over a contact $3$-dimensional Lie algebra \cite[Theorem 1.4]{angella-bazzoni-parton}, we should mention that this construction, in the Hermitian case, corresponds to the known relation between lcK and Sasakian structures, see \cite{andrada-origlia-Vaisman}. In particular, the subclass of Vaisman Lie algebras can be constructed in this way.
\end{rmk}

\begin{landscape}

\appendix
\section{Table of lcK structures on four-dimensional Lie algebras}\label{app:table}

\def\arraystretch{1.4}
\tiny
\begin{table}[h]
\centering
{\resizebox{1.4\textwidth}{!}{
\begin{tabular}{clcl|lll|l|l}
\hline\toprule
   Lie algebra & & cplx structure & & \multicolumn{3}{|c|}{non-K\"ahler lcK structure} & Vaisman & note \\    
   & & & & Lee form $\theta$ & positive form $\Omega$ & parameters & & \\    
\toprule
    \hline
$\mathbb{R}^4$ & & & & no lcK structure & & & no & torus \\
    \hline\hline
\multirow{3}{*}{$\mathfrak{gl}_2$} & & \multirow{3}{*}{$J_{1,\mu}$} & \multirow{2}{*}{ $\mu \in \mathbb R \setminus \{0\}$ } &$-\mu_{1}  e^{4}$ & $\omega_{12}  e^{12} + \omega_{13}  e^{13} + \omega_{23} \mu_{1}  e^{14} + \omega_{23}  e^{23} + \frac{1}{2} \, \omega_{12} \mu_{1}  e^{24} - \frac{1}{2} \, \omega_{13} \mu_{1}  e^{34}$ & $\omega_{12}\geq\omega_{13} > 0$, $\omega_{23}\geq0$, $\omega_{12}\omega_{13}-\omega_{23}^2 > 0$ & $\omega_{23}=0$, $\omega_{12}=\omega_{13}$ & \\
\cline{5-9}
&&&& $\theta_{4}  e^{4}$ & $\omega_{12}  e^{12} + \omega_{12}  e^{13} - \frac{1}{2} \, \omega_{12} \theta_{4}  e^{24} + \frac{1}{2} \, \omega_{12} \theta_{4}  e^{34}$ & $\theta_4\neq-\mu_1$, $\omega_{12}>0$, $\frac{\theta_4}{\mu_1} < 0$ & \multirow{2}{*}{always} & \\
\cline{4-7}
&&& $\mu \in \mathbb C \setminus (\mathbb R \cup \sqrt{-1} \mathbb R)$ & $\theta_{4}  e^{4}$ & $\omega_{12}  e^{12} + \omega_{12}  e^{13} - \frac{1}{2} \, \omega_{12} \theta_{4}  e^{24} + \frac{1}{2} \, \omega_{12} \theta_{4}  e^{34}$ & $\omega_{12}>0$, $\frac{\theta_4}{\mu_1} < 0$ & & \\
    \hline
\multirow{2}{*}{$\mathfrak{u}_2$} & & $J_{a,b}$ & $a\neq0$, $b\neq0$ & $\theta_4e^4$ & $\omega_{23}\theta_4e^{14}+\omega_{23}e^{23}$ & $\theta_4\neq0$, $\omega_{23}<0$, $\frac{\theta_4}{b}>0$ & \multirow{2}{*}{always} & \\
\cline{3-7}
 & & $J_{0,b}$ & $b\neq0$ & $\theta_4e^4$ & $\omega_{23}\theta_4e^{14}+\omega_{23}e^{23}$ & $\theta_4\not\in\{0,-\frac{1}{b}\}$, $\omega_{23}<0$, $\frac{\theta_4}{b}>0$ & & \\
    \hline\hline
$\mathfrak{rh}_3$ & & & & $-e^4$ & $\sigma e^{12}+\sigma e^{34}$ & $\sigma>0$ & always & primary Kodaira \cite{belgun} \\
    \hline
$\mathfrak{rr}_{3,0}$ & & & & $\delta e^1 + \frac{\sigma}{\delta} e^3$ & $\frac{\delta(\delta+1)}{\sigma}e^{12}+e^{14}-e^{23}+\frac{\sigma}{\delta^2}e^{34}$ & $\delta>0$, $\sigma>0$ & always & \\
    \hline
$\mathfrak{rr}_{3,1}$ & & & & $-2e^1$ & $\sigma e^{14}-e^{23}$ & $\sigma>0$ & never & \\
    \hline
$\mathfrak{rr}^\prime_{3,0}$ & & & & no lcK structures & & & no & hyperelliptic \cite{belgun} \\
    \hline
\multirow{2}{*}{$\mathfrak{rr}^\prime_{3,\gamma}$} & \multirow{2}{*}{$\gamma>0$} & $J_1$ & & $-2\gamma e^1$ & $\sigma e^{14}+e^{23}$ & $\sigma>0$ & \multirow{2}{*}{never} & \\
    \cline{3-7}
& & $J_2$ & & $-2\gamma e^1$ & $\sigma e^{14}-e^{23}$ & $\sigma>0$ & & \\
    \hline
\multirow{3}{*}{$\mathfrak{r}_2\mathfrak{r}_2$} & & & & $-e^3$ & $\omega_{12}(e^{12}-e^{14}+e^{23})+\omega_{34} e^{34}$ & $\omega_{34}>\omega_{12}>0$ & \multirow{3}{*}{never} & \\
\cline{3-7}
& & & & $-e^1$ & $\omega_{12} e^{12} + \omega_{34}(-e^{14}+e^{23}+e^{34})$ & $\omega_{12}>\omega_{34}>0$ & \\
\cline{3-7}
& & & & $\sigma e^1+\tau e^3$ & $\mu ( \frac{1+\sigma}{\tau} e^{12} + e^{14} - e^{23} + \frac{\tau+1}{\sigma} e^{34}) $ & $\sigma\tau\neq 0$, $\sigma+\tau\neq-1$, $\mu\neq0$, $\frac{\mu(1+\sigma)}{\tau}>0$, $\frac{\mu(1+\tau)}{\sigma}>0$, $\frac{\sigma+\tau+1}{\sigma\tau}>0$ & \\
    \hline
\multirow{5}{*}{$\mathfrak{r}^\prime_{2}$} & & \multirow{3}{*}{$J_1$} & & $\theta_1 e^1+\theta_2 e^2$ & $\omega_{13}(e^{13}+\frac{\theta_2}{\theta_1}(e^{14}+e^{23})+\frac{\theta_2^2-\theta_1}{\theta_1(\theta_1+1)}e^{24})$ & $\omega_{13}>0$, $\theta_1\neq-1$, $\theta_1\neq0$, $\theta_2>0$, $\frac{\theta_2^2-\theta_1}{\theta_1(\theta_1+1)}>\frac{\theta_2^2}{\theta_1^2}$ & \multirow{5}{*}{never} & \\
    \cline{5-7}
& & & & $-2e^1$ & $\omega_{12}(e^{12}+e^{34})+\omega_{13}(e^{13}+e^{24})$ & $\omega_{13}>0$, $\omega_{13}^2-\omega_{12}^2>0$, $\omega_{12}>0$ & & \\
    \cline{5-7}
& & & & $\theta_1e^1$ & $\omega_{24}(-(\theta_1+1)e^{13}+e^{24})$ & $\omega_{24}>0$, $\theta_1+1<0$ & & \\
    \cline{3-7}
& & \multirow{2}{*}{$J_2^{a,b}$} & $(a,b)\neq (0,1)$ & $-2e^1$ & $\omega_{12}e^{12}+e^{34}$ & $b\omega_{12}>0$ & & \\
    \cline{4-7}
& & & $(a,b) = (0,1)$ & $-2e^1$ & $\sigma e^{12}+\omega_{34} e^{34}$ & $\sigma>0$ & & \\
	\hline
$\mathfrak{r}_{4,1}$ & & & & no lcK structure & & & no & \\
	\hline
$\mathfrak{r}_{4,\alpha,1}$ & $\alpha\not\in\{0,1\}$ & & & $-2e^4$ & $e^{13}+\sigma e^{24}$ & $\sigma<0$ & never & \\
	\hline
$\mathfrak{r}_{4,\alpha,\alpha}$, $\hat{\mathfrak{r}}_{4,-1}$ & $\alpha\not\in\{-1,0,1\}$ & & & $-2\alpha e^4$ & $\sigma e^{14}+e^{23}$ & $\sigma<0$ & never & \\
	\hline
\multirow{2}{*}{$\mathfrak{r}^\prime_{4,0,\delta}$} & \multirow{2}{*}{$\delta>0$} & $J_1$ & & no lcK structure & & & \multirow{2}{*}{no} & \\
	\cline{3-7}
& & $J_2$ & & no lcK structure & & & \\
	\hline
\multirow{2}{*}{$\mathfrak{r}^\prime_{4,\gamma,\delta}$} & \multirow{2}{*}{$\delta>0$, $\gamma\neq0$} & $J_1$ & & $-2\gamma e^4$ & $\sigma e^{14}+e^{23}$ & $\sigma<0$ & \multirow{2}{*}{never} & \multirow{2}{*}{$\gamma=-\frac{1}{2}$: Inoue $S^0$ \cite{tricerri, belgun}} \\
	\cline{3-7}
& & $J_2$ & & $-2\gamma e^4$ & $\sigma e^{14}-e^{23}$ & $\sigma<0$ & \\
	\hline
\multirow{2}{*}{$\mathfrak{d}_{4}$} & & $J_1$ & & $-e^4$ & $-e^{13}+\sigma e^{24}$ & $\sigma<0$ & \multirow{2}{*}{never} & \multirow{2}{*}{Inoue $S^+$ \cite{tricerri, belgun}} \\
	\cline{3-7}
& & $J_2$ & & no lcK structure & & & \\
	\hline
\multirow{3}{*}{$\mathfrak{d}_{4,1}$} & & & & $-e^4$ & $\sigma e^{14}+e^{23}$ & $\sigma>0$ & \multirow{3}{*}{never} & \\
	\cline{3-7}
& & & & $e^2+\theta_4e^4$ & $\left( \frac{\omega_{23} {\left(\theta_{4} + 1\right)}}{\theta_{2}^{2}} \right)  e^{12} + \left( \frac{{\left(\omega_{12} \theta_{2} - \omega_{23}\right)} {\left(\theta_{4} + 1\right)}}{\theta_{2}^{2}} \right)  e^{14}$ & $\omega_{23}> 0$, $(\omega_{12}\theta_2 - \omega_{23})(\theta_4 + 1)>0$,  & \\
& & & & & $+ \frac{\omega_{23}}{\theta_{2}^{2}}  e^{23} + \left( -\frac{\omega_{23} {\left(\theta_{4} + 1\right)}}{\theta_{2}^{2}} \right)  e^{34}$ & $(\omega_{12}\theta_2-\omega_{23}-\omega_{23}(\theta_4+1))\omega_{23}(\theta_4+1)>0$ & \\
\hline
\multirow{3}{*}{$\mathfrak{d}_{4,\frac{1}{2}}$} & & $J_1$ & & $\theta_4e^4$ & $\tau (e^{12}-(\sigma+1)e^{34}))$ & $\tau>0$, $\sigma+1>0$, $\theta_4\neq0$ & \multirow{3}{*}{never} & \\
 \cline{3-7}
& & $J_2$ & & $-\frac{3}{2}e^4$ & $\sigma(e^{12}+\frac{1}{2}e^{34}))$ & $\sigma<0$ & \\
 \cline{3-7}
& & $J_3$ & & $-\frac{3}{2}e^4$ & $\sigma e^{14}-e^{23}$ & $\sigma<0$ & \\
	\hline
\multirow{2}{*}{$\mathfrak{d}_{4,\lambda}$} & \multirow{2}{*}{$\lambda\not\in\{\frac{1}{2},1\}$} & $J_1$ & & $(\lambda-2)e^4$ & $\sigma e^{14}
+e^{23}$ & $\sigma>0$, $ \lambda \neq 2$ & \multirow{2}{*}{never} & \\
\cline{3-7}
& & $J_2$ & & $-(1+\lambda)e^4$ & $e^{13}+\sigma e^{24}$ & $(\lambda-1)\sigma>0$  & \\
	\hline
\multirow{2}{*}{$\mathfrak{d}^\prime_{4,0}$} & & $J_2$ & & $\mu e^4$ & $\sigma (e^{12}-\mu e^{34})$ & $\sigma>0$, $\mu<0$ & \multirow{2}{*}{always} & \multirow{2}{*}{secondary kodaira \cite{belgun}} \\
\cline{3-7}
& & $J_3$ & & $\mu e^4$ & $\sigma (e^{12}-\mu e^{34})$ & $\sigma>0$, $\mu>0$ & \\
	\hline
\multirow{4}{*}{$\mathfrak{d}^\prime_{4,\delta}$} & \multirow{4}{*}{$\delta>0$} & $J_1$ & & $\mu e^4$ & $\sigma (e^{12}-(\delta+\mu) e^{34})$ & $\sigma<0$, $\delta+\mu<0$, $\mu\neq0$ & \multirow{4}{*}{never} & \\
\cline{3-7}
& & $J_2$ & & $\mu e^4$ & $\sigma (e^{12}-(\delta+\mu) e^{34})$ & $\sigma>0$, $\delta+\mu<0$, $\mu\neq0$ & \\
\cline{3-7}
& & $J_3$ & & $\mu e^4$ & $\sigma (e^{12}-(\delta+\mu) e^{34})$ & $\sigma>0$, $\delta+\mu>0$, $\mu\neq0$ & \\
\cline{3-7}
& & $J_4$ & & $\mu e^4$ & $\sigma (e^{12}-(\delta+\mu) e^{34})$ & $\sigma<0$, $\delta+\mu>0$, $\mu\neq0$ & \\
\hline
$\mathfrak{h}_4$ & & & & no lcK structure & & & no & \\
\hline\bottomrule
\end{tabular}
}}
\caption{Locally conformally K\"ahler non-K\"ahler structures on $4$-dimensional Lie algebras, up to complex automorphisms of the Lie algebra.}
\label{table:lck-solvable}
\end{table}
\normalsize
\end{landscape}

\end{document}